\documentclass[a4paper]{article}

\usepackage[pages=all, color=black, position={current page.south}, placement=bottom, scale=1, opacity=1, vshift=5mm]{background}

\usepackage[margin=1in]{geometry} 

\usepackage{amsmath}
\usepackage{amsthm}
\usepackage{amssymb}

\usepackage[utf8]{inputenc}
\usepackage{hyperref}

\usepackage{verbatim,xcolor}
\usepackage{todonotes}
\usepackage{bm}
\usepackage{url}
\usepackage{float, subcaption}
\usepackage{multirow}
\usepackage{algorithmicx}

\usepackage[sort&compress,numbers,square]{natbib}
\theoremstyle{plain}
\newtheorem{theorem}{Theorem}[section]

\newtheorem{lemma}[theorem]{Lemma}

\theoremstyle{definition}

\newtheorem{remark}[theorem]{Remark}

\usepackage{graphicx, color}
\graphicspath{{fig/}}
\usepackage{algorithm, algpseudocode}
\usepackage{mathrsfs}

\usepackage{lipsum}
\numberwithin{equation}{section}

\newcommand{\ba}{{\bf a}}
\newcommand{\bb}{{\bf b}}

\newcommand{\bd}{{\bf d}}

\newcommand{\bx}{{\bf x}}
\newcommand{\bn}{{\bf n}}
\newcommand{\bu}{{\bf u}}
\newcommand{\bv}{{\bf v}}
\newcommand{\bvc}{\bv^C}
\newcommand{\bvd}{\bv^D}
\newcommand{\bw}{{\bf w}}
\newcommand{\bwc}{\bw^C}
\newcommand{\bwd}{\bw^D}

\newcommand{\bV}{{\bf V}}

\newcommand{\bC}{{\bf C}}
\newcommand{\bD}{{\bf D}}
\newcommand{\bQ}{{\bf Q}}
\newcommand{\bbf}{{\bf f}}

\newcommand{\bzero}{\bf{0}}
\newcommand{\bphi}{\boldsymbol \phi}

\newcommand{\bPhi}{\boldsymbol \Phi}
\newcommand{\baE}{{\ba_{\mathcal{E}}}}
\newcommand{\bdE}{{\bd_{\mathcal{E}}}}

\newcommand{\Pih}{\Pi_h}
\newcommand{\Pz}{\mathcal{P}_0}

\newcommand{\cR}{\mathcal{R}}

\newcommand{\PiC}{\Pi^{C}}

\def\ljump{{[\![}}
\def\rjump{{]\!]}}

\newcommand{\jump}[1]{[ #1 ]}
\newcommand{\avg}[1]{\{ #1\}}
\newcommand{\Hone}{H^1(\Omega)}

\newcommand{\Honed}{[H^1(\Omega)]^d}
\newcommand{\Honezd}{[H^1_0(\Omega)]^d}
 \newcommand{\Ltwo}{L^2(\Omega)}
 \newcommand{\Ltwoz}{L^2_0(\Omega)}
 \newcommand{\Vh}{\bV_h}
 \newcommand{\Qh}{\bQ_h}
  \newcommand{\Ch}{\bC_h}
  \newcommand{\Dh}{\bD_h}
\newcommand{\Th}{\mathcal{T}_h}
\newcommand{\K}{K}
\newcommand{\Eh}{\mathcal{E}_h}
\newcommand{\Eho}{\mathcal{E}_h^o}
\newcommand{\Ehb}{\mathcal{E}_h^b}
\newcommand{\bne}{\bn_e}
\newcommand{\bnK}{\bn_\K}
\newcommand{\cD}{\mathcal{D}}
\newcommand{\norm}[1]{\lVert #1\rVert}
\newcommand{\enorm}[1]{\lVert #1\rVert_{\mathcal{E}}}
\newcommand{\snorm}[1]{|#1|}

\newcommand{\buh}{\bu_h}
\newcommand{\ph}{p_h}
\newcommand{\chiu}{\chi_\bu}
\newcommand{\xiu}{\xi_\bu}
\newcommand{\chip}{\chi_p}
\newcommand{\xip}{\xi_p}
\newcommand{\bbuh}{\overline{\bu}_h}

\newcommand{\half}{\frac{1}{2}}

\title{Pressure-robust enriched Galerkin methods for the Stokes equations\footnote{Submitted to Journal of Computational and Applied Mathematics in 2023.}}
\author{Xiaozhe Hu,\thanks{Department of Mathematics, Tufts University, Medford, MA 02155
(\texttt{xiaozhe.hu@tufts.edu})}
\and Seulip Lee,\thanks{Department of Mathematics, University of Georgia, Athens, GA 30602 (\texttt{seulip.lee@uga.edu})}
\and Lin Mu,\thanks{Department of Mathematics, University of Georgia, Athens, GA 30602 (\texttt{linmu@uga.edu})}
\and Son-Young Yi\thanks{Department of Mathematical Sciences, University of Texas at El Paso, El Paso, TX 79968 (\texttt{syi@utep.edu})}}

\date{
}

\begin{document}
	\maketitle
	
	\begin{abstract}
		In this paper, we present a pressure-robust enriched Galerkin (EG) scheme for solving the Stokes equations, which is an enhanced version of the EG scheme for the Stokes problem proposed in [S.-Y. Yi, X. Hu, S. Lee, J. H. Adler, An enriched Galerkin method for the Stokes equations, \textit{Computers and Mathematics with Applications} 120 (2022) 115–131]. The pressure-robustness is achieved by employing a velocity reconstruction operator on the load vector on the right-hand side of the discrete system. An a priori error analysis proves that the velocity error is  independent of the pressure and viscosity. We also propose and analyze a perturbed version of our pressure-robust EG method that allows for the elimination of the degrees of freedom corresponding to the discontinuous component of the velocity vector via static condensation. The resulting method can be viewed as a stabilized $H^1$-conforming $\mathbb{P}_1$-$\mathbb{P}_0$ method.
Further, we consider efficient block preconditioners whose performances are independent of the viscosity.
The theoretical results are confirmed through various numerical experiments in two and three dimensions.
	\end{abstract}

	

\section{Introduction}\label{sec: intro}
We consider the Stokes equations for modeling incompressible viscous flow in an open and bounded domain $\Omega\subset \mathbb{R}^{d}$, $d=2,3$, 
with simply connected Lipschitz boundary $\partial \Omega$: Find the fluid velocity $\bu:\Omega\rightarrow\mathbb{R}^d$ and the pressure $p:\Omega\rightarrow\mathbb{R}$ such that
\begin{subequations}\label{sys: governing}
\begin{alignat}{2}
-\nu\Delta \bu + \nabla p & = \bbf && \quad \text{in } \Omega, \label{eqn: governing1} \\
\nabla\cdot \bu &= 0 && \quad \text{in } \Omega,  \label{eqn: governing2}\\
\bu & = 0 && \quad \text{on }\partial \Omega, \label{eqn: governing3}
\end{alignat}
\end{subequations}
where $\nu>0$ is the fluid viscosity, and $\bbf$ is a given body force.

Various  finite element methods (FEMs) have been applied to solve the Stokes problem based on the velocity-pressure formulation \eqref{sys: governing}, including conforming and non-conforming mixed FEMs \cite{TaylorHood73,BernardiRaugel85, CrouzeixRaviart73, ScottVogelius83}, discontinuous Galerkin (DG) methods 
\cite{HansboLarson08,GiraultEtAl05}, weak Galerkin methods \cite{MuWYZ18,Mu20}, and enriched Galerkin methods \cite{ChaabaneEtAl18,YiEtAl22-Stokes}.
It is well-known that the finite-dimensional solution spaces must satisfy the inf-sup stability condition \cite{Ladyzhenskaya69,Babuska73,Brezzi74} for the well-posedness of the discrete problem regardless of what numerical method is used. Therefore, there has been extensive research to construct inf-sup stable pairs for the Stokes equations in the last several decades. 
Some classical $H^1$-conforming stable pairs include Taylor-Hood, Bernardi-Raugel, and MINI elements \cite{GiraultRaviart86}.

Though the inf-sup condition is crucial for the well-posedness of the discrete problem, it does not always guarantee an accurate solution. Indeed, many inf-sup stable pairs are unable to produce an accurate velocity solution when the viscosity is very small.  More precisely, such pairs produce the velocity solution whose  error bound depends on the pressure error and is inversely proportional to the viscosity, $\nu$.
Hence, it is important to develop numerical schemes whose velocity error bounds are independent of pressure and viscosity, which we call {\it pressure-robust} schemes.

There have been three major directions to develop pressure-robust schemes.
The first direction is based on employing a divergence-free velocity space since, then, one can separate the velocity error from the pressure error. 
Recall, however, that the velocity space of classical low-order inf-sup stable elements do not satisfy the divergence-free (incompressibility) condition strongly. 
One way to develop a divergence-free velocity space is to take the curl of an $H^2$-conforming finite element space \cite{ArnoldQin92,FalkNeilan13,GuzmanNeilan14(1)}, from which the pressure space is constructed by taking the divergence operator. Though this approach provides the desired pressure-robustness, it requires many degrees of freedom (DoFs).
Another way, which has received much attention lately, is to employ an $H(\text{div})$-conforming velocity space \cite{CockburnKS07,WangYe07,ChenWZ14,GuzmanNeilan14(1),GuzmanNeilan14(2),ChenWang2018a}.
However, to take account of lack of regularity, the tangential continuity of the velocity vector  across the inter-elements has to be imposed   either strongly \cite{GuzmanNeilan14(1),GuzmanNeilan14(2)} or weakly \cite{CockburnKS07,WangYe07,ChenWZ14,ChenWang2018a}.
The second direction is based on the grad-div stabilization \cite{Olshanskii04,JenkinsJLR14}, which is derived by adding a modified incompressibility condition to the continuous momentum equation. 
This stabilization technique reduces the pressure effect in the velocity error estimate but not completely eliminates it.
The third direction, which we consider in the present work, is based on employing a velocity reconstruction operator \cite{Linke12}. 
In this approach, one reconstructs ($H(\text{div})$-nonconforming) discrete velocity test functions by mapping them into an $H(\text{div})$-conforming space. These reconstructed velocity test functions are used in the load vector (corresponding to the body force $\bbf$) on the right-hand side of the discrete system while the original test functions are used in the stiffness matrix. 
This idea has been successfully explored in various numerical methods \cite{Linke14,LinkeMerdon16,GaugerLS19,Mu20,MuYZ21(1),MuYZ21(2),WangMWH21,LiZikatanov22}.

Our main goal is to develop and analyze a pressure-robust finite element method requiring minimal number of degrees of freedom.
To achieve this goal, we consider as the base method the enriched Galerkin (EG) method proposed for the Stokes equations with mixed boundary conditions \cite{YiEtAl22-Stokes}. This EG method employs the piecewise constant space for the pressure. As for the velocity space, it employs the  linear Lagrange space,
enriched by some discontinuous, piecewise linear, and mean-zero vector functions. This enrichment space requires only one DoF per element. Indeed, any function $\bv$ in the velocity space has a unique decomposition of the form $\bv = \bv^C + \bv^D$, where $\bv^C$ belongs to the linear Lagrange space and $\bv^D$ belongs to the discontinuous linear enrichment space. To take account of the non-conformity of the velocity space, an interior penalty discontinuous Galerkin (IPDG) bilinear form is adopted. 
This method is one of the cheapest inf-sup stable methods with optimal convergence rates for the Stokes problem. However, like many other inf-sup stable methods, this EG method produces pressure- and viscosity-dependent velocity error.  In order to derive a pressure-robust EG method, we propose to take the velocity reconstruction approach mentioned above. Specifically, the velocity test functions are mapped to the first-order Brezzi-Douglas-Marini space, whose resulting action is equivalent to preserving the continuous component $\bv^C$
and mapping only the discontinuous component $\bv^D$ to the lowest-order Raviart-Thomas space. 
This operator is applied to the velocity test functions on the right-hand side of the momentum equation, therefore the resulting stiffness matrix is the same as the one generated by the original EG method in \cite{YiEtAl22-Stokes}. By this simple modification in the method, we can achieve pressure-robustness without compromising the optimal convergence rates, which has been proved mathematically and demonstrated numerically.

Though the proposed pressure-robust EG method is already a very cheap and efficient method, we seek to reduce the computational costs  even more by exploring a couple of strategies. 
First, we developed a perturbed pressure-robust EG method, where a sub-block in the stiffness matrix corresponding to the discontinuous component, $\buh^D$, of the velocity solution $\buh = \buh^C + \buh^D$  is replaced by a diagonal matrix.  This modification allows for the elimination of the DoFs corresponding to $\buh^D$ from the discrete system via static condensation. The method corresponding to the condensed linear system can be viewed as a new stabilized $H^1$-conforming $\mathbb{P}_1$-$\mathbb{P}_0$ method (see \cite{rodrigoNewStabilizedDiscretizations2018} for a similar approach).
For an alternative strategy, we designed fast linear solvers for the pressure-robust EG method and its two variants. Since the stiffness matrix remains unchanged, the fast linear solvers designed for the original EG method in \cite{YiEtAl22-Stokes} can be applied to the new pressure-robust EG method and its two variants with only slight modifications. The performance of these fast solvers was investigated through some numerical examples.

The remainder of this paper is organized as follows: Section~\ref{sect:preliminaries} introduces some preliminaries and notations, which are useful in the later sections.
In Section~\ref{sec: EG}, we recall  the standard EG method \cite{YiEtAl22-Stokes} and introduce our pressure-robust EG method.
Then, well-posedness and error estimates are proved in Section~\ref{sec: error}, which reveal the pressure-robustness of our new EG scheme.
In Section~\ref{sec:perturb-EG}, a perturbed pressure-robust EG scheme is presented, and its numerical analysis is discussed. 
In Section~\ref{sec:nume_examples}, we verify our theoretical results and demonstrate the efficiency of the proposed EG methods through two-and three-dimensional numerical examples.
We conclude our paper by summarizing the contributions of our new EG methods and discuss future research in Section~\ref{sec:conclusions}.

\section{Notation and Preliminaries}\label{sect:preliminaries}
We first present some notations and preliminaries that will be useful for the rest of this paper. 
For any open domain $\cD \in \mathbb{R}^d$, where $d=2, 3$, we use the standard notation $H^s(\cD)$ for Sobolev spaces, where $s$ is a positive real number. The Sobolev norm and semi-norm associated with $H^s(\cD)$ are denoted by $\norm{\cdot}_{s, \cD}$  and $\snorm{\cdot}_{s, \cD}$, respectively.  In particular, when $s = 0$, $H^0(\cD)$ coincides with $L^2(\cD)$ and its associated norm  will be denoted by $\norm{\cdot}_{0, \cD}$.
Also, $(\cdot, \cdot)_\cD$ denotes the $L^2$-inner product on $\cD$. If $\cD = \Omega$, we drop $\cD$ in the subscript. We extend these definitions and notation{s} to vector{-} and tensor-valued Sobolev spaces in a straightforward manner.  
Additionally, $H^1_0(\Omega)$ and $\Ltwoz$ are defined by
\[
H^1_0(\Omega) = \{ z \in \Hone \mid z = 0 \text{  on  }  \partial \Omega \}, \quad
\Ltwoz = \{ z \in \Ltwo \mid (z, 1)= 0  \}.
\]
We finally introduce a Hilbert space
\begin{equation*}
    H(\text{div},\Omega):=\{\bv\in [L^2(\Omega)]^d:\text{div}\;\bv\in L^2(\Omega)\}
\end{equation*}
with an associated norm
\begin{equation*}
\norm{\bv}_{H(\text{div},\Omega)}^2:=\norm{\bv}_{0}^2+\norm{\text{div}\;\bv}_{0}^2.
\end{equation*}

To define our EG methods, we consider a shape-regular mesh $\Th$ on $\Omega$, consisting of triangles in two dimensions and tetrahedra in three dimensions. Let $h_\K$ be the diameter of $\K \in \Th$. Then, 
the characteristic mesh size $h$ is defined by
$h = {\max}_{\K \in \Th}{h_\K}$.
Denote  by $\Eho$ the collection of all the interior edges/faces in $\Th$ and by $\Ehb$ the boundary edges/faces . Then, $\Eh = \Eho \cup \Ehb$ is the collection of all edges/faces in the mesh $\Th$. 
On each $\K \in \Th$, let $\bnK$ denote the unit outward normal vector on $\partial \K$. Each interior edge/face $e \in \Eho$ is shared by two neighboring elements, that is, $ e = \partial \K^+ \cap \partial \K^-$ for some $\K^+, \K^- \in \Th$. We associate one unit normal vector $\bne$ with $e \in \Eho$, which is assumed to be oriented from $\K^+$ to $\K^-$. If $e$ is a boundary edge/face, then $\bne$ is the unit outward normal vector to $\partial \Omega$.

We now define broken Sobolev spaces on $\Th$ and $\Eh$. 
The broken Sobolev space $H^s(\Th)$  on the mesh $\Th$ is defined by 
\[
H^s(\Th) = \{ v \in \Ltwo \mid v|_{\K} \in H^s(\K) \ \forall \K \in \Th \},
\]
which is equipped with a broken Sobolev norm
\[
\norm{v}_{s, \Th} = \left( \sum_{\K \in \Th} \norm{v}_{s, \K}^2 \right)^\half.
\]
Similarly, we define the space $L^2(\Eh)$ on $\Eh$  and its associated norm
\[
\norm{v}_{0, \Eh} =  \left( \sum_{e \in \Eh} \norm{v}_{0, e}^2 \right)^\half.
\]
Also, the $L^2$-inner product on $\Eh$ is denoted by $\langle \cdot, \cdot \rangle_{\Eh}$.

Finally, we define the average and jump operators, which will be needed to define the EG methods.  For any $v \in H^s(\Th)$ with $s > 1/2$,  let $v^\pm$ be the trace of $v|_{\K^\pm}$  on $e = \partial \K^+ \cap \partial \K^-$. Then, the average and jump of $v$ along $e$, denoted by $\avg{\cdot}$ and $\jump{\cdot}$, are defined by
\[
\avg{v} = \half( v^+ + v^-), \quad \jump{v} = v^+ - v^- \quad \text{on } e \in \Eho.
\]
On a boundary edge/face,  
\[
\avg{v} =v, \quad \jump{v} = v \quad \text{on } e \in \Ehb.
\]
These definitions can be naturally extended to vector- or tensor-valued functions.

\section{Pressure-Robust Enriched Galerkin Method}\label{sec: EG}
In this section, we introduce our new pressure-robust EG method for the Stokes equations. 
First, we can derive the following weak formulation of the governing equations \eqref{sys: governing} in a standard way:
Find $(\bu, p) \in \Honezd \times \Ltwoz$ such that
\begin{subequations}\label{sys: weak}
\begin{alignat}{2}
\nu   (\nabla \bu, \nabla \bv) - (\nabla \cdot \bv, p) & = (\bbf, \bv) && \quad \forall \bv \in  \Honezd, \label{eqn: weak1} 
\\
(\nabla \cdot \bu, q) & = 0 && \quad \forall q \in  \Ltwoz. \label{eqn: weak2}
\end{alignat}
\end{subequations}
In this paper, we consider the homogeneous Dirichlet boundary condition for simplicity. However, in the case of a non-homogeneous Dirichlet boundary condition, this  weak problem can be easily modified by changing the solution space for the velocity $\bu$ to reflect the non-homogeneous boundary condition.

\subsection{Standard EG method for the Stokes problem \cite{YiEtAl22-Stokes}}
Our pressure-robust EG method is designed based upon the EG method originally proposed in \cite{YiEtAl22-Stokes}, which we will refer to as the {\it standard EG method} for the Stokes equations.
Therefore, in this section, we first present the standard EG method to lay a foundation for the presentation of the new pressure-robust method.

For any integer $k \ge 0$, $\mathbb{P}_k(\cD)$ denotes the set of polynomials  defined on $\cD \subset \mathbb{R}^d$ whose total degree is less than or equal to $k$. Let
\[
\Ch = \{\bv \in \Honezd \mid \bv|_{\K} \in [\mathbb{P}_1(\K)]^d \ \forall \K \in \Th \}
\]
and 
\[
\Dh = \{\bv \in L^2(\Omega) \mid \bv|_{\K} = c (\bx - \bx_\K),  c \in \mathbb{R} \ \forall \K \in \Th\},
\]
where $\bx_\K$ is the centroid of the element $\K \in \Th$. 
Then, our EG finite element space for the velocity is defined by
\[
\Vh = \Ch \oplus \Dh.
\]
Therefore, any $\bv \in \Vh$ has a unique decomposition $\bv = \bv^C + \bv^D$ such that $\bv^C \in \Ch$ and $\bv^D \in \Dh$.
On the other hand, we employ the mean-zero, piecewise constant space for the pressure. That is, the pressure space is defined by
\[
\Qh = \{ q \in \Ltwoz \mid q \in \mathbb{P}_0(\K) \ \forall \K \in \Th \}.
\]

The standard EG scheme introduced in \cite{YiEtAl22-Stokes} reads as follows:
\begin{algorithm}[H]
\caption{Standard EG (\texttt{ST-EG}) method } \label{alg:EG}
Find $( \bu_h, \ph) \in \Vh \times \Qh $ such that
\begin{subequations}\label{sys: stand-eg}
\begin{alignat}{2}
\ba(\buh,\bv)  - \bb(\bv, \ph) &= (\bbf,  \bv) &&\quad \forall \bv \in\Vh, \label{eqn: stand-eg1}\\
\bb(\buh,q) &= 0 &&\quad \forall q\in \Qh,  \label{eqn: stand-eg2}
\end{alignat}
\end{subequations}
where 
\begin{subequations}\label{sys: bilinear}
\begin{align}
\ba(\bv,\bw) &:= \nu \big((\nabla \bv,\nabla \bw)_{\Th} -  \langle \avg{\nabla\bv} \bn_e, \jump{\bw} \rangle_{\Eh} \nonumber\\
&\qquad\qquad\qquad\qquad- \langle \avg {\nabla\bw} \bn_e,\jump{\bv}\rangle_{\Eh}  +  \rho \langle h_e^{-1}\jump{\bw},
\jump{\bv}\rangle_{\Eh} \big), \label{eqn: bia} \\
\bb(\bw,q)&:= (\nabla\cdot\bw, q)_{\Th} - \langle \jump{\bw}\cdot\bn_e,\avg{q} \rangle_{\Eh}.  \label{eqn: bib}
\end{align}
\end{subequations}
Here, $\rho >0$ is a penalty parameter and $h_e = |e|^{1/(d-1)}$, where $|e|$ is the length/area of the edge/face $e \in \Eh$. 
\end{algorithm}

\subsection{Pressure-robust EG method}
In order to define a pressure-robust EG scheme, we use a velocity reconstruction operator $\cR: \Vh \to  H(\text{div},\Omega)$
that maps any function $\bv \in \Vh$ into the  Brezzi-Douglas-Marini space of index $1$, denoted by BDM$_1$ \cite{BrezziFortin91}. Specifically, for any $\bv  \in \Vh$, $\cR \bv \in\text{BDM}_1$ is defined by
\begin{subequations}\label{sys: BDM}
\begin{alignat}{2}
\int_e (\cR \bv) \cdot\bn_e  p_1\  ds & = \int_e \avg{\bv}\cdot\bn_e p_1 \ ds 
 && \quad \forall p_1 \in \mathbb{P}_1(e), \ \forall e \in \Eho,  \\
\int_e (\cR \bv) \cdot\bn_e  p_1\  ds & = 0  && \quad \forall p_1 \in \mathbb{P}_1(e), \ \forall e \in \partial \Omega.
\end{alignat}
\end{subequations}

\begin{remark}
The above definition of $\cR$ is useful in numerical analysis of the new method. In practice, however, we reconstruct only the discontinuous component of each vector based on the following observation:
For any $\bv = \bv^C + \bv^D \in \Vh$, the operator $\cR$ defined in \eqref{sys: BDM} keeps the continuous component $\bv^C$ the same while mapping $\bv^D$ into the lowest-order Raviart-Thomas space \cite{raviart2006mixed}, denoted by $\text{RT}_0$. Indeed,  $\cR \bv^D \in \text{RT}_0$ satisfies
\begin{equation}\label{eqn: RT}
\int_e \cR \bv^D \cdot\bn_e \ ds = \int_e \avg {\bv^D} \cdot\bn_e \ ds \quad  \forall e \in \Eh.
\end{equation}
\end{remark}
Using the definition of the operator $\cR$, we can easily prove the following lemma.
\begin{lemma} The bilinear form $\bb(\bw,q) $ can be written using the velocity reconstruction operator $\cR$ as follows:
\begin{equation}\label{eqn: bib-equi}
\bb(\bw,q) =  (\nabla \cdot \cR \bw, q)_{\Th}.
\end{equation}
\end{lemma}

We are now ready to define our pressure-robust EG method using the velocity reconstruction operator $\cR$ as below.
\begin{algorithm}[H]
\caption{Pressure-robust EG (\texttt{PR-EG}) method}
\label{alg:EG-PR}
Find $( \bu_h, \ph) \in \Vh \times \Qh $ such that
\begin{subequations}\label{sys: eg}
\begin{alignat}{2}
\ba(\buh,\bv)  - \bb(\bv, \ph) &= (\bbf, \cR \bv), &&\quad \forall \bv \in\Vh, \label{eqn: eg1}\\
\bb(\buh,q) &= 0, &&\quad \forall q\in \Qh,  \label{eqn: eg2}
\end{alignat}
\end{subequations}
where the bilinear forms $\ba(\cdot, \cdot)$ and $\bb(\cdot, \cdot)$ are the same as in \eqref{sys: bilinear}.
\end{algorithm}

As the bilinear forms $\ba(\cdot, \cdot)$ and $\bb(\cdot, \cdot)$ are the same as in \eqref{sys: bilinear}, the \texttt{PR-EG} method in Algorithm~\ref{alg:EG-PR} and the \texttt{ST-EG} method in Algorithm~\ref{alg:EG}
have the same stiffness matrix. However, this seemingly simple modification on the right-hand side vector significantly enhances the performance of the \texttt{ST-EG} method, as will be shown in the error estimates and numerical experiments.

\section{Well-Posedness and Error Estimates}\label{sec: error}
We first establish the inf-sup condition \cite{BrezziFortin91} and prove a priori error estimates for the \texttt{PR-EG} method, Algorithm~\ref{alg:EG-PR}. To this end, we employ the following mesh-dependent norm in $\Vh$:
\[
\enorm{\bv} = \left(\norm{\nabla \bv}_{0, \Th}^2 +  \rho \norm{h_e^{-\half}  \jump{\bv}}_{0, \Eh}^2\right)^\half.
\]
Also, the analysis relies on the interpolation operator $\Pih: \Honed \to \Vh$, defined in \cite{YiLeeZikatanov21, YiEtAl22-Stokes}, and the local $L^2$-projection $\Pz: \Hone \to \Qh$. 
Here, we only state their useful properties and error estimates without proof.
\begin{lemma} 
There exists an interpolation operator $\Pih: \Honed \to \Vh$  such that
\begin{subequations}\label{sys: Pih}
\begin{alignat}{2}
& (\nabla\cdot(\bv - \Pih \bv), 1)_{\K} = 0 &&  \quad \forall \K \in \Th,\ \forall \bv \in \Honed, \label{eqn: Pih_prop} \\
& |\bv - \Pih \bv | _{j} \leq C h^{m-j} |\bv|_{m} && \quad 0 \leq j \leq m \leq 2, \ \forall \bv \in [H^2(\Omega)]^d, \label{eqn: Pih_err} \\
& \enorm{\bv - \Pih \bv} \leq C h \norm{\bv}_2 &&\quad 0 \leq j \leq m \leq 2, \ \forall \bv \in [H^2(\Omega)]^d,  \label{eqn: Pih_energy_err}\\
& \enorm{\Pih \bv} \leq C \snorm{\bv}_1 && \quad \forall \bv \in \Honezd.  \label{eqn: Pih_energy}
\end{alignat}
\end{subequations}
\end{lemma}
Also, the local $L^2$-projection $\Pz: \Hone \to \Qh$ satisfies
\begin{subequations}\label{sys: Pz}
\begin{alignat}{2}
& (w - \Pz w, 1) = 0 &\quad \forall w \in \Hone, \label{eqn: Pz_prop} \\
& \norm{ w - \Pz w}_0 \leq C h \norm{w}_1  &\quad \forall w \in H^1(\Omega).  \label{eqn: Pz_err}
\end{alignat}
\end{subequations}

The inf-sup condition and the coercivity and continuity of the bilinear form $\ba(\cdot, \cdot)$ can be proved 
following the same lines as the proofs of Lemma 4.5 in \cite{YiEtAl22-Stokes}. Therefore, we only state the results here. 
\begin{lemma} Provided that $\rho$ is large enough, there exists a constant $\alpha >0$, independent of $h$, such that
\begin{equation}\label{eqn: infsup}
\underset{\substack{\bv \in \Vh \\ \bv \ne 0}}{sup} \frac{\bb(\bv, q)}{\enorm{\bv}} \ge \alpha \norm{q}_0 \quad \forall q \in \Qh. 
\end{equation}
\end{lemma}

\begin{lemma}  \label{lem: a}
Provided that the penalty parameter $\rho$ is large enough, there exist positive constants $\kappa_1$ and $\kappa_2$, independent of $h$ and $\nu$, such that
\begin{alignat}{2}
\ba(\bv, \bv) & \ge \kappa_1 \nu \enorm{\bv}^2 && \quad \forall \bv \in \Vh, \label{eqn: ba_coer} \\
|\ba(\bv, \bw)| & \leq \kappa_2 \nu \enorm{\bv}\enorm{\bw} && \quad \forall \bv, \bw \in \Vh. \label{eqn: ba_conti}
\end{alignat}
\end{lemma}
We now have the following well-posedness result of the \texttt{PR-EG} method.
\begin{theorem}
There exists a unique solution $(\buh, \ph) \in \Vh \times \Qh$ to the \texttt{PR-EG} method,
provided that the penalty term $\rho >0$ is large enough.
\end{theorem}
\begin{proof}
Thanks to the finite-dimensionality of $\Vh$ and $\Qh$, it suffices to show the uniqueness of the solution. 
Suppose there exist two solutions $(\bu_{h,1}, p_{h,1})$ and $(\bu_{h,2}, p_{h,2})$ to the \texttt{PR-EG} method, and let $\bw_h = \bu_{h,1} - \bu_{h,2}$ and $r_h = p_{h,1} -p_{h,2}$. 
Then, it is trivial to see that 
\begin{subequations}\label{sys: diff}
\begin{alignat}{2}
\ba(\bw_h,\bv)  - \bb(\bv, r_h) &= 0 &&\quad \forall \bv \in \Vh, \label{eqn: diff_u}\\
\bb(\bw_h, q) & = 0 &&\quad \forall q \in \Qh. \label{eqn: diff_p}
\end{alignat}
\end{subequations}
Taking $\bv = \bw_h$ and $q =r_h$ in the above system \eqref{sys: diff} and adding the two equations, we obtain
\[
\ba(\bw_h,\bw_h) = 0.
\]
Then, the coercivity of $\ba(\cdot, \cdot)$ yields $\bw_h = \bzero$. 
On the other hand, $\bw_h = \bzero$ reduces \eqref{eqn: diff_u} to
\[
\bb(\bv, r_h) = 0 \quad \forall \bv \in \Vh.
\]
Then, we get $r_h =0$ by using the inf-sup condition in \eqref{eqn: infsup}. 
Hence, the proof is complete. 
\end{proof}

In order to facilitate our error analysis, we consider an elliptic projection $\bbuh$ of the true solution $\bu$ defined as follows: 
\begin{equation}\label{eqn: ellip_proj}
\ba(\bbuh, \bv) = \ba(\bu, \bv) + \nu(\Delta \bu, \bv - \cR \bv)_{\Th} \quad \forall \bv \in \Vh.
\end{equation}
Then, let us introduce the following notation: 
\begin{equation}\label{eqn: nota}
 \chiu = \bu - \bbuh, \quad \xiu = \bbuh - \buh, 
 \quad \chip= p - \Pz p, \quad \xip = \Pz p  - \ph.
 \end{equation}
Using the above notation, we have
\[
\bu - \buh = \chiu + \xiu,\quad p -\ph = \chip + \xip.
\]
As we already have an error estimate for $\chip$ in \eqref{eqn: Pz_err}, we only need to establish error estimates for $\chiu$, $\xiu$, and $\xip$. Let us first prove the following lemma, which will be used for an estimate for $\chiu$. 
\begin{lemma} For any $\bv \in \Vh$, we have
\begin{equation}\label{eqn: cR-err}
\norm{ \cR \bv - \bv}_{0} \leq C h \enorm{\bv}.
\end{equation}
\end{lemma}
\begin{proof}
For any $\phi_\K \in \text{RT}_0(\K)$, using the standard scaling argument~\cite[p. 554]{RobertsThomas91}, we have 
\begin{equation}\label{eqn: rt_bd}
\norm{\bphi_\K}_{0, \K}
\leq C \sum_{e \in \partial \K} h_{e}^{\half} \norm{\bphi_\K \cdot \bn_\K}_{0, e}
\end{equation}
for some constant $C >0$ independent of $h$. 
Apply \eqref{eqn: rt_bd} to $\bphi_\K = (\cR \bv^D - \bv^D)|_{\K} \in \text{RT}_0(\K)$ and note that 
\[
\avg{\bv^D} - \bv^D = \pm \half \jump{\bv^D}.
\]
Then, 
\begin{align*}
\norm{ \cR \bv - \bv}_{0}^2 & =  \norm{\cR \bv^D - \bv^D}_{0, \Th}^2 
\leq C   \norm{h_{e}^\half \jump{\bv^D} \cdot \bne}_{0, \Eh}^2
\leq C   h^2  \norm{h_{e}^{-\half} \jump{\bv^D}}_{0, \Eh}^2 \\
& = C   h^2  \norm{h_{e}^{-\half} \jump{\bv}}_{0, \Eh}^2
\leq  C h^2 \enorm{\bv}^2.
\end{align*}
\end{proof}

\begin{lemma} Let $\bbuh$ be the solution of the elliptic problem \eqref{eqn: ellip_proj}.
{Assuming the true velocity solution $\bu $ belongs to $[H^2(\Omega)]^d$,} the following error estimate holds true: 
\begin{equation}\label{eqn: ellip_proj_err}
\enorm{\bu - \bbuh} \leq C h \norm{\bu}_2. 
\end{equation}
\end{lemma}
\begin{proof}
We already have
$
\enorm{\bu - \Pih \bu} \leq C h \norm{\bu}_2
$
from \eqref{eqn: Pih_energy_err}. To bound the error $\xi_{\bbuh} :=\bbuh - \Pih \bu$, 
subtract $\ba(\Pih \bu, \bv)$ from both sides of \eqref{eqn: ellip_proj} and take $\bv = \xi_{\bbuh}$  to obtain
\begin{equation*}
\ba( \xi_{\bbuh}, \xi_{\bbuh} ) = \ba(\bu - \Pih \bu , \xi_{\bbuh} ) + \nu(\Delta \bu, \xi_{\bbuh}   - \cR \xi_{\bbuh} )_{\Th}.  \\
\end{equation*}
Then, we bound the left-hand side by using the coercivity of $\ba(\cdot, \cdot)$ and the right-hand side by the continuity of $\ba(\cdot, \cdot)$, along with the Cauchy-Schwarz inequality and \eqref{eqn: cR-err}. Then, we get
\begin{equation*}
\kappa_1 \nu \enorm{\xi_{\bbuh}}^2 \leq \kappa_2 \nu \enorm{\bu - \Pih \bu}\enorm{\xi_{\bbuh}} + \nu \norm{\bu}_2 \norm{\xi_{\bbuh}   - \cR \xi_{\bbuh} }_0 \leq  C h \nu \norm{\bu}_2 \enorm{\xi_{\bbuh}}. 
\end{equation*}
Therefore, \eqref{eqn: ellip_proj_err} follows from the triangle inequality combined with \eqref{eqn: Pih_energy_err}.
\end{proof}

Now, we consider the following lemmas  to derive auxiliary error equations.

\begin{lemma} For all $\bv \in \Vh$, we have
\begin{equation}\label{eqn: cons_b2}
(\nabla p, \cR \bv)_{\Th}  = - \bb(\bv, \Pz p). 
\end{equation}
\end{lemma}
\begin{proof}
Recall that $\cR \bv \cdot \bnK$ is continuous across the interior edges (faces) and zero on the boundary edges (faces), and $\nabla \cdot \cR \bv \in \mathbb{P}_0(\K)$ on each $\K \in \Th$. Therefore, using integration by parts, the regularity of the pressure $p$,  and \eqref{eqn: bib-equi}, we obtain 
\begin{align*}
(\nabla p, \cR \bv)_{\Th} & = \sum_{\K \in \Th} \langle p, \cR \bv \cdot \bnK \rangle_{\partial \K} - (p, \nabla \cdot  \cR \bv)_{\Th} \\
& = - (\Pz p, \nabla \cdot  \cR \bv)_{\Th} = - \bb(\bv, \Pz p).
\end{align*}
\end{proof}

\begin{lemma} The auxiliary errors $\xiu$ and $\xip$ satisfy the following equations: 
\begin{subequations}
\begin{alignat}{2}
\ba(\xiu, \bv) - \bb(\bv, \xip) &= 0  &&\quad \forall \bv \in \Vh, \label{eqn: aux_err1} \\
\bb(\xiu, q) &= - \bb(\chiu, q) && \quad \forall q \in \Qh. \label{eqn: aux_err2}
\end{alignat}
\end{subequations}
\end{lemma}
\begin{proof}
To derive \eqref{eqn: aux_err1}, take the $L^2$-inner product of \eqref{eqn: governing1} with  $\cR \bv$ for $\bv \in \Vh$ to get 
\begin{equation}\label{eqn: der1}
- \nu (\Delta \bu, \bv)_{\Th} + (\nabla p, \cR \bv)_{\Th} = (\bbf, \cR \bv)_{\Th} - \nu(\Delta \bu, \bv - \cR \bv)_{\Th}.
\end{equation}
Note that standard integration by parts and the regularity of $\bu$ yield the following identity:
 $$- \nu (\Delta \bu, \bv)_{\Th}  = \ba(\bu, \bv) \quad \forall \bv \in \Vh.$$
Further, using  \eqref{eqn: cons_b2} for the second term on the left side of \eqref{eqn: der1}, we obtain
\[
\ba(\bu, \bv)  - \bb(\bv, \Pz p) = (\bbf, \cR \bv)_{\Th} - \nu(\Delta \bu, \bv - \cR \bv)_{\Th}. 
\]
Then, subtract \eqref{eqn: eg1} from the above equation and use the definition of the elliptic projection $\bbuh$ in \eqref{eqn: ellip_proj} to obtain the first error equation \eqref{eqn: aux_err1}.  
Next, to prove the second error equation \eqref{eqn: aux_err2}, first note
$$(\nabla \cdot \bu, q)_{\Th}  = \bb(\bu, q) \quad \forall q \in \Qh.$$ 
The rest of the proof is  straightforward. 
\end{proof}

\begin{lemma} 
We have the following error bound: 
\begin{subequations}
\begin{equation}\label{eqn: L2_xip}
\norm{\xip}_0 \leq C \nu \enorm{\xiu}.
\end{equation}
\end{subequations}
\end{lemma}
\begin{proof}
We obtain the desired bound by using \eqref{eqn: infsup}, \eqref{eqn: aux_err1}, and \eqref{eqn: ba_conti} as follows:
\begin{equation*}
\norm{\xip}_0   \leq  \frac{1}{\alpha}  \underset{\substack{\bv \in \Vh \\ \bv \ne 0}}{sup} \frac{\bb(\bv, \xip)}{\enorm{\bv}} 
= \frac{1}{\alpha}  \underset{\substack{\bv \in \Vh \\ \bv \ne 0}}{sup}\frac{\ba(\xiu, \bv)}{\enorm{\bv}}
\leq C \nu \enorm{\xiu}. 
\end{equation*}
\end{proof}
The last two lemmas  lead us to the following auxiliary error estimates. 
\begin{lemma}
We have the following error estimates for $\xiu$ and $\xip$:
\begin{subequations}
\begin{align}
\enorm{\xiu} & \leq C h \norm{\bu}_2,  \label{eqn: err1}\\ 
\norm{\xip}_0 & \leq C \nu  h \norm{\bu}_2.\label{eqn: err2}
\end{align}
\end{subequations}
\end{lemma}
\begin{proof}
Taking $\bv = \xiu$ and $q = \xip$ in \eqref{eqn: aux_err1} and \eqref{eqn: aux_err2}, respectively, and adding the two resulting equations, we get
\[
\ba(\xiu, \xiu) = - \bb(\chiu, \xip).  
\]
Then, we bound the left-hand side from below using the coercivity of $\ba(\cdot, \cdot)$ and the right-hand side by using \eqref{eqn: Pih_prop}, the trace inequality, \eqref{eqn: Pih_err}, and \eqref{eqn: L2_xip}:
\begin{align*}
\kappa_1 \nu \enorm{\xiu}^2 & \leq \ba(\xiu, \xiu) = - \bb(\chiu, \xip) \\
 & = - \langle\jump{ \chiu} \cdot\bn_e,\avg{\xip} \rangle_{\Eh} 
 \leq C \frac{1}{h} \norm{\chiu}_0 \norm{\xip}_0 
 \leq C \nu  h \norm{\bu}_2  \enorm{\xiu},
\end{align*}
from which \eqref{eqn: err1} follows. 
On the other hand, the error bound \eqref{eqn: err2} is an immediate consequence of \eqref{eqn: L2_xip} and \eqref{eqn: err1}.
\end{proof}

We now state the main theorem of this section. 
\begin{theorem} Assuming the true solution, $(\bu, p)$, of the Stokes problem \eqref{sys: governing} belongs to $ [H^2(\Omega)]^d \times (L^2_0(\Omega) \cap H^1(\Omega))$, the solution $(\buh, \ph)$ to our \texttt{PR-EG} method satisfies the following error estimates:
\begin{align*}
\enorm{\bu - \buh} & \leq C h \norm{\bu}_2, \\ 
\norm{p - \ph}_0 & \leq C h (\nu   \norm{\bu}_2 + \norm{p}_1),
\end{align*}
provided that the penalty parameter $\rho$ is large enough.
\end{theorem}

\section{Perturbed Pressure-Robust EG Method}\label{sec:perturb-EG}
In this section, we consider a perturbed version of the
\texttt{PR-EG} method, Algorithm~\ref{alg:EG-PR},
where a perturbation is introduced in the bilinear form $\ba: \Vh \times \Vh \mapsto \mathbb{R}$. The perturbed bilinear form is denoted by $\ba^D(\cdot,\cdot)$ and assumed to satisfy the following coercivity and continuity conditions: 
\begin{alignat}{2}
	\ba^D(\bv, \bv) & \ge \kappa^D_1 \nu \enorm{\bv}^2 && \quad \forall \bv \in \Vh, \label{eqn: baD_coer} \\
	|\ba^D(\bv, \bw)| & \leq \kappa^D_2 \nu \enorm{\bv}\enorm{\bw} && \quad \forall \bv, \bw \in \Vh, \label{eqn: baD_conti}
\end{alignat}
where $\kappa_1^D$ and $\kappa_2^D$ are positive constants, independent of $h$ and $\nu$.  
Additionally, we make the following assumption on the bilinear form $\ba^D(\cdot, \cdot)$: 
\begin{equation}
\ba^D(\bv^C, \bw) = \ba(\bv^C,\bw) \quad \forall \, \bv^C \in \Ch, \ \forall \bw \in \Vh. \label{eqn: perturb-assump-cons} \tag{A1}
\end{equation}	
This assumption provides a consistency property of $\ba^D(\cdot, \cdot)$ and is useful to maintain the optimal convergence rate of the perturbed pressure-robust EG scheme.

Using a perturbed bilinear form $\ba^D(\cdot, \cdot)$, a perturbed version of the \texttt{PR-EG} method is defined as follows:
\begin{algorithm}[H]
\caption{Perturbed pressure-robust EG (\texttt{PPR-EG}) method }\label{alg:EG-PR-Elimination}
Find $( \bu_h, \ph) \in \Vh \times \Qh $ such that
\begin{subequations}\label{sys: perturb-eg}
	\begin{alignat}{2}
		\ba^D(\buh,\bv)  - \bb(\bv, \ph) &= (\bbf, \cR \bv), &&\quad \forall \bv \in\Vh, \label{eqn: perturb-eg1}\\
		\bb(\buh,q) &= 0, &&\quad \forall q\in \Qh.  \label{eqn: perturb-eg2}
	\end{alignat}
\end{subequations}
Here, a specific choice of $\ba^D(\cdot,\cdot)$ will be presented later in \eqref{eqn: baD_DG}.
\end{algorithm}

In what follows, we establish the well-posedness and optimal-order error estimates of the \texttt{PPR-EG}
method with a general choice of $\ba^D(\cdot,\cdot)$ satisfying the coercivity~\eqref{eqn: baD_coer}, continuity~\eqref{eqn: baD_conti}, and Assumption~\eqref{eqn: perturb-assump-cons}. Then, at the end of this section, we introduce one specific choice of $\ba^D(\cdot, \cdot)$ that allows for the elimination of the degrees of freedom corresponding to the DG component of the velocity vector $\buh$, {\it i.e.,} $\buh^D \in \Dh$, via static condensation.

\subsection{Well-posedness and error estimates}
Based on the coercivity~\eqref{eqn: baD_coer} and continuity~\eqref{eqn: baD_conti} of the perturbed bilinear form $\ba^D(\cdot, \cdot)$, together with the inf-sup condition~\eqref{eqn: infsup}, we have the following well-posedness result of the \texttt{PPR-EG} method, Algorithm~\ref{alg:EG-PR-Elimination}.
\begin{theorem} \label{thm:perturb-eg-wellposed}
	There exists a unique solution $(\buh, \ph) \in \Vh \times \Qh$ to the \texttt{PPR-EG} method,
	provided that the penalty term $\rho >0$ is large enough.
\end{theorem}
An a priori error analysis for the \texttt{PPR-EG} method
can be done in a similar fashion to that of the \texttt{PR-EG} method, Algorithm~\ref{alg:EG-PR}. For the sake of brevity, we will only indicate the necessary changes in the previous analysis. First, we define the following elliptic projection $\bbuh$ of the true solution $\bu$ using $\ba^D(\cdot, \cdot)$: 
\begin{equation}\label{eqn: perturb_ellip_proj}
	\ba^D(\bbuh, \bv) = \ba(\bu, \bv) + \nu(\Delta \bu, \bv - \cR \bv)_{\Th} \quad \forall \bv \in \Vh.
\end{equation}
\begin{lemma} 
	Let $\bbuh$ be the solution of the elliptic problem \eqref{eqn: perturb_ellip_proj}.  If Assumption~\eqref{eqn: perturb-assump-cons} holds {and $\bu \in [H^2(\Omega)]^d$}, then  the following error estimate holds true: 
	\begin{equation}\label{eqn: perturb_ellip_proj_err}
		\enorm{\bu - \bbuh} \leq C h \norm{\bu}_2. 
	\end{equation}
\end{lemma}
\begin{proof}
	Let $\PiC: [H^1(\Omega)]^d \mapsto \Ch$ be the usual linear Lagrange interpolation operator.  Then, we have
	$
	\enorm{\bu - \PiC \bu} = \snorm{\bu - \PiC \bu}_1  \leq C h \norm{\bu}_2.  
	$
	Let  $\xi_{\bbuh} :=\bbuh - \PiC \bu$. Then, 
	subtracting $\ba^D(\PiC \bu, \bv)$ from both sides of \eqref{eqn: perturb_ellip_proj}, taking $\bv = \xi_{\bbuh}$, and using the fact that $\ba^D(\PiC \bu, \bv) = \ba(\PiC \bu, \bv)$ by Assumption~\eqref{eqn: perturb-assump-cons}, we obtain
	\begin{equation*}
		\ba^D( \xi_{\bbuh}, \xi_{\bbuh} ) = \ba(\bu - \PiC \bu , \xi_{\bbuh} ) + \nu(\Delta \bu, \xi_{\bbuh}   - \cR \xi_{\bbuh} )_{\Th}.  
	\end{equation*}
	Then,  using the coercivity of $\ba^D(\cdot, \cdot)$ and the continuity of $\ba(\cdot, \cdot)$, along with the Cauchy-Schwarz inequality and \eqref{eqn: cR-err},  we get
	\begin{equation*}
		\kappa^D_1 \nu \enorm{\xi_{\bbuh}}^2 \leq \kappa_2 \nu \enorm{\bu - \PiC \bu}\enorm{\xi_{\bbuh}} + \nu \norm{\bu}_2 \norm{\xi_{\bbuh}   - \cR \xi_{\bbuh} }_0 \leq  C h \nu \norm{\bu}_2 \enorm{\xi_{\bbuh}}. 
	\end{equation*}
	Therefore, \eqref{eqn: perturb_ellip_proj_err} follows from the triangle inequality $\enorm{\bu - \bbuh} \leq \enorm{\bu - \PiC \bu} + \enorm{\xi_{\bbuh}}$.
\end{proof}
The rest of the error analysis follows the same lines as in Section~\ref{sec: error}, hence the details are omitted here.
We state the optimal error estimates for the \texttt{PPR-EG} method
in the following theorem:

\begin{theorem} \label{thm:perturb_EG_error}
Assume that the solution $(\bu, p) \in [H^2(\Omega)]^d \times {(L^2_0(\Omega) \cap \Hone)}$ and that the perturbed bilinear form $\ba^D$ satisfies \eqref{eqn: baD_coer}, \eqref{eqn: baD_conti}, and \eqref{eqn: perturb-assump-cons}. Then, the solution $(\buh, \ph)$ to the \texttt{PPR-EG} method satisfies the following error estimates  provided that the penalty parameter $\rho$ is large enough.
	\begin{align*}
		\enorm{\bu - \buh} & \leq C h \norm{\bu}_2, \\ 
		\norm{p - \ph}_0 & \leq C h (\nu   \norm{\bu}_2 + \norm{p}_1). 
	\end{align*}
\end{theorem}

\subsection{Perturbed bilinear form $\ba^D$}

In this section, we introduce one particular choice of the perturbed bilinear form $\ba^D(\cdot, \cdot)$, which allows for the elimination of the DG
component of the velocity vector via static condensation.

For any $\bv \in \Vh$, we have a unique decomposition
	$\bv = \bvc + \bvd$,
where $\bvc \in \Ch$ and $\bvd \in \Dh$. Therefore, we have, for $\bv, \ \bw \in \Vh$,
\begin{equation*}
\ba(\bv, \bw) = \ba(\bvc, \bwc) + \ba(\bvc,\bwd) + \ba(\bvd,\bwc) + \ba(\bvd,\bwd).
\end{equation*}
Denote the basis of $\Dh$ by $\{\bPhi_{K} \}_{K \in \Th}$ and write $\bvd = \sum_{K \in \Th} v_K \bPhi_K$ and  $\bwd = \sum_{K \in \Th} w_K \bPhi_K$.
Then,
\begin{equation*}
	\ba(\bvd,\bwd) = \sum_{K, \ K' \in \Th} v_K w_{K'} \ba(\bPhi_K, \bPhi_{K'}).
\end{equation*}
Define a bilinear form $\bd: \Dh \times \Dh \mapsto \mathbb{R}$ by
\begin{equation*}
	\bd(\bvd,\bwd) := \sum_{K \in \Th} v_K w_K \ba(\bPhi_K, \bPhi_K).
\end{equation*}
We define a perturbed bilinear form $\ba^D(\cdot, \cdot)$ of $\ba(\cdot, \cdot)$ by replacing $\ba(\bvd,\bwd)$ with $\bd(\bvd,\bwd)$. That is, for $\bv, \ \bw \in \Vh$,
\begin{equation}\label{eqn: baD_DG}
	\ba^D(\bv, \bw) := \ba(\bvc,\bwc) + \ba(\bvc,\bwd) + \ba(\bvd,\bwc) + \bd(\bvd,\bwd). 
\end{equation}
Note that, for any $\bvc \in \Ch$, 
\begin{equation*}
	\ba^D(\bvc, \bw) = \ba(\bvc,\bwc) + \ba(\bvc,\bwd) = \ba(\bvc, \bw), \quad \forall \, \bw \in \Vh, 
\end{equation*}
which verifies that the bilinear form $\ba^D(\cdot, \cdot)$ defined in \eqref{eqn: baD_DG} satisfies Assumption \eqref{eqn: perturb-assump-cons}.

Next, we want to prove the coercivity and continuity of $\ba^D(\cdot, \cdot)$ on $\Dh$ with respect to $\enorm{\cdot}$. This will be done in several steps. We first introduce a bilinear form $\baE$ corresponding to the mesh-dependent norm $\enorm{\cdot}$:
\begin{equation*}
\baE(\bv, \bw) : = \nu \left ( (\nabla \bv,\nabla \bw)_{\Th}  +  \rho \langle h_e^{-1} \ljump\bw \rjump,\ljump\bv \rjump\rangle_{\Eh} \right ), 
\end{equation*}
from which we immediately see that 
\begin{equation*}
	\nu \enorm{\bv}^2 = \baE(\bv,\bv). 
\end{equation*}
Also, for $\bvd, \ \bwd \in \Dh$, we have
\begin{equation*}
	\baE(\bvd,\bwd) = \sum_{K, \ K' \in \Th} u_K v_{K'} \baE(\bPhi_K, \bPhi_{K'}).
\end{equation*}
Then, we define another bilinear form $\bdE: \Dh \times \Dh \mapsto \mathbb{R}$ as follows:
\begin{equation*}
	\bdE(\bvd,\bwd) := \sum_{K \in \Th} u_K v_K \baE(\bPhi_K, \bPhi_K).
\end{equation*}
Note, from the definitions, that both bilinear forms $\baE$ and $\bdE$ are symmetric and positive definite (SPD). Also, thanks to  the coercivity and continuity of the bilinear form $\ba(\cdot, \cdot)$ on $\Vh$ with respect to $\enorm{\cdot}$ proved in Lemma~\ref{lem: a}, we have the following spectral equivalence results:
\begin{align}
\kappa_1 \baE(\bvd, \bvd) \leq \ba(\bvd, \bvd) \leq \kappa_2 \baE(\bvd, \bvd), \label{eqn: equivalent_a_aE} \\
\kappa_1 \bdE(\bvd, \bvd) \leq \bd(\bvd, \bvd) \leq \kappa_2 \bdE(\bvd, \bvd) \label{eqn: equivalent_d_dE}.
\end{align}
Therefore,  if $\baE(\bvd, \bvd)$ and $\bdE(\bvd, \bvd)$ are spectrally equivalent, the spectral equivalence of  $\ba(\bvd, \bvd)$ and $\bd(\bvd, \bvd)$ follows directly.  In fact, the spectral equivalence between $\baE(\bvd, \bvd)$ and $\bdE(\bvd, \bvd)$ has been shown in~\cite{YiLeeZikatanov21}.

\begin{lemma}\cite{YiLeeZikatanov21}
	There exist positive constants $C_1$ and $C_2$, depending only on the dimension $d$, the shape regularity of the mesh, and the penalty parameter $\rho$, such that 
	\begin{equation} \label{eqn: equivalent_aE_dE}
		C_1 \bdE(\bvd, \bvd) \leq \baE(\bvd, \bvd) \leq C_2 \bdE(\bvd, \bvd) 
		\quad \forall \bvd \in \Dh.
	\end{equation}
\end{lemma}
\begin{proof}
The spectral equivalence follows from~\cite[Equations (5.23), (5.24), and Lemma 5.6]{YiLeeZikatanov21} and the fact that $\bvd \in \Dh \subset \Vh$.
\end{proof}

Therefore, we can conclude that $\ba(\bvd, \bvd)$ and $\bd(\bvd, \bvd)$ are spectral equivalent as stated in the following lemma.
\begin{lemma}
	There exist positive constants $C_1$ and $C_2$ such that 	
	\begin{equation}\label{eqn: equivalent_a_d}
		C_1 \bd(\bvd, \bvd) \leq \ba(\bvd, \bvd) \leq C_2 \bd(\bvd, \bvd) \quad \forall \bvd \in \Dh.  
	\end{equation}
Here, $C_1$ and $C_2$ depend only on the dimension $d$, the shape regularity of the mesh, the penalty parameter $\rho$, and $\kappa_1$ and $\kappa_2$. 
\end{lemma}
\begin{proof}
The desired result  is an immediate consequence of ~\eqref{eqn: equivalent_a_aE}, \eqref{eqn: equivalent_d_dE}, and~\eqref{eqn: equivalent_aE_dE}.
\end{proof}

We are now ready to show that the perturbed bilinear form $\ba^D(\cdot, \cdot)$ defined in~\eqref{eqn: baD_DG} satisfies the desired coercivity and continuity conditions.
\begin{lemma}
The bilinear form $\ba^D(\cdot, \cdot)$ defined in~\eqref{eqn: baD_DG} satisfies the coercivity condition~\eqref{eqn: baD_coer} and the continuity condition~\eqref{eqn: baD_conti} with some positive constants $\kappa_1^D$ and $\kappa_2^D$.  Here, the constants $\kappa_1^D$ and $\kappa_2^D$ depend only on the dimension $d$, the shape regularity of the mesh, the penalty parameter $\rho$, and the constants $\kappa_1$ and $\kappa_2$.
\end{lemma}
\begin{proof}
As we deal with SPD bilinear forms here,  the conclusion follows directly from the definitions of $\ba(\cdot, \cdot)$ and $\ba^D(\cdot, \cdot)$, the spectral equivalence result~\eqref{eqn: equivalent_a_d} on the enrichment space $\Dh$, and several applications of triangular and Cauchy-Schwarz inequalities. 
\end{proof}

So far, we have verified the coercivity and continuity conditions, ~\eqref{eqn: baD_coer} and ~\eqref{eqn: baD_conti}, and Assumption~\eqref{eqn: perturb-assump-cons} for the perturbed bilinear form $\ba^D(\cdot, \cdot)$ defined in \eqref{eqn: baD_DG}. Therefore, the resulting 
\texttt{PPR-EG} method in Algorithm~\ref{alg:EG-PR-Elimination} is well-posed and converges at the optimal rate as proved in Theorem~\ref{thm:perturb_EG_error}.

\subsection{Elimination of the DG component of the velocity}
The \texttt{PR-EG} method in Algorithm~\ref{alg:EG-PR} results in a linear system in the following matrix form:
\begin{equation}\label{eqn: stiff-A}
\begin{pmatrix}
\sf{A}_{DD} & \sf{A}_{DC} & \sf{G}_{D} \\
\sf{A}_{CD} & \sf{A}_{CC} & \sf{G}_{C} \\
\sf{G}^{\top}_{D}  & \sf{G}^{\top}_{C}    & \sf{0}
 \end{pmatrix}
 \begin{pmatrix}
    \sf{U}^D \\ \sf{U}^C \\ \sf{P}
    \end{pmatrix}
    =
    \begin{pmatrix}
        \sf{f}_D \\ \sf{f}_C \\ \sf{0}
 \end{pmatrix},
 \end{equation}
where $\ba(\bvd,\bwd) \mapsto \sf{A}_{DD}$, $\ba(\bvd,\bwc) \mapsto \sf{A}_{DC}$, $\ba(\bvc,\bwd) \mapsto \sf{A}_{CD}$, $\ba(\bvc,\bwc) \mapsto \sf{A}_{CC}$, $\bb(\bvd, p) \mapsto \sf{G}_{D}$, $\bb(\bvc, p) \mapsto \sf{G}_{C}$, $(\bbf, \cR \bvd) \mapsto \sf{f}_D$, and $(\bbf, \bvc) \mapsto \sf{f}_C$.  

The coefficient matrix in \eqref{eqn: stiff-A} will be denoted by $\mathcal{A}$.
On the other hand, the \texttt{PPR-EG} method, Algorithm~\ref{alg:EG-PR-Elimination},  basically replaces $\sf{A}_{DD}$ in \eqref{eqn: stiff-A} with a diagonal matrix $\sf{D}_{DD} := \texttt{diag}(\sf{A}_{DD})$, which was resulted from $\bd(\bvd,\bwd)$. Therefore, its stiffness matrix is as follows:
\begin{equation*}
	\mathcal{A}^D :=
	\begin{pmatrix}
		\sf{D}_{DD} & \sf{A}_{DC} & \sf{G}_{D} \\
		\sf{A}_{CD} & \sf{A}_{CC} & \sf{G}_{C} \\
		\sf{G}^{\top}_{D}  & \sf{G}^{\top}_{C}    & \sf{0}
	\end{pmatrix}.
\end{equation*}
Indeed, the diagonal block $\sf{D}_{DD} $ allows us to eliminate the DoFs corresponding to the DG component of the velocity vector. Specifically,  we can obtain the following two-by-two block form via static condensation:
\begin{equation*}
	\mathcal{A}^E :=
	\begin{pmatrix}
		\sf{A}_{CC} - \sf{A}_{CD}\sf{D}_{DD}^{-1}\sf{A}_{DC} & \sf{G}_{C} -  \sf{A}_{CD}\sf{D}_{DD}^{-1} \sf{G}_{D}\\
	  \sf{G}^{\top}_{C} - \sf{G}^{\top}_{D}  \sf{D}_{DD}^{-1} \sf{A}_{DC}  & -  \sf{G}^{\top}_{D}  \sf{D}_{DD}^{-1}\sf{G}_{D}
	\end{pmatrix} :=
    \begin{pmatrix}
    	\sf{A}^E_u & \sf{G}^E \\
    	(\sf{G}^E)^{\top} & -\sf{A}^E_p
    \end{pmatrix}.
\end{equation*}
The method corresponding to this condensed linear system is well-posed as it is obtained from the well-posed \texttt{PPR-EG} method via state condensation (see \cite[Theorem 3.2]{adler2020robust} for a similar proof). Moreover, this method has the same DoFs as the $H^1$-conforming $\mathbb{P}_1$-$\mathbb{P}_0$ method for the Stokes equations. Therefore, this method can be viewed as a stabilized $\mathbb{P}_1$-$\mathbb{P}_0$ scheme for the Stokes equations, where a stabilization term appears in every sub-block.  Similar stabilization techniques have been studied in~\cite{RodrigoGasparHuZikatanov2016a, rodrigoNewStabilizedDiscretizations2018}. 
We emphasize that this new stabilized $\mathbb{P}_1$-$\mathbb{P}_0$ scheme is not only stable but also pressure-robust. 
Then, the algorithm to find $\sf{U} = (\sf{U}^D, \sf{U}^C)^{\top}$ and $\sf{P}$ is summarized in Algorithm~\ref{alg:EG-elimination-2}.
 \alglanguage{pseudocode}
\begin{algorithm}[ht!]
\caption{Condensed pressure-robust EG (\texttt{CPR-EG}) method }\label{alg:EG-elimination-2}
\begin{algorithmic}[1]
\State Compute $\mathsf{f}^E = (\sf{f}_C - \sf{A}_{CD}\sf{D}_{DD}^{-1}\sf{f}_D, -\sf{G}_D^{\top}\sf{D}_{DD}^{-1}\sf{f}_D)$
\State Solve $\mathcal{A}^E \mathsf{x}^E = \mathsf{f}^E$ for $\mathsf{x}^E$
\State Set $(\sf{U}^C, \sf{P})= \mathsf{x}^E$
\State Compute $\sf{U}^D = \sf{D}_{DD}^{-1}(\sf{f}_D - \sf{A}_{DC}\sf{U}^C - \sf{G}_D\sf{P})$
\end{algorithmic}
\end{algorithm}

\subsection{Block preconditioners} \label{sec:block-prec}
In this subsection, we discuss block preconditioners for solving the linear systems resulted from the three pressure-robust EG algorithms, {\it i.e.}, the \texttt{PR-EG}, \texttt{PPR-EG}, and \texttt{CPR-EG} methods. We mainly follow the general framework developed in~\cite{loghin2004analysis,mardal2011preconditioning,adler2017robust,adler2020robust} to design robust block preconditioners. The main idea is based on the well-posedness of the proposed EG discretizations.

As mentioned in Section~\ref{sec: EG}, the \texttt{ST-EG} method in Algorithm~\ref{alg:EG} and the \texttt{PR-EG} method in Algorithm~\ref{alg:EG-PR} have the same stiffness matrices but different right-hand-side vectors. 
Therefore, the block preconditioners developed for the \texttt{ST-EG} method in~\cite{YiEtAl22-Stokes} can be directly applied for the \texttt{PR-EG} method.
For the sake of completeness, we recall those preconditioners here as follows.
\begin{equation*}
\mathcal{B}_D =
\begin{pmatrix}
	\sf{A}_{u} & \sf{0} \\
	 \sf{0}    & \nu^{-1}\sf{M}_p
\end{pmatrix}^{-1},
\quad 
\mathcal{B}_L =
\begin{pmatrix} 
	\sf{A}_{u} & \sf{0} \\
	\sf{G}^{\top}   & \nu^{-1}\sf{M}_p
\end{pmatrix}^{-1},
\quad
\mathcal{B}_U =
\begin{pmatrix}
	\sf{A}_{u} &  \sf{G} \\
	\sf{0}   & \nu^{-1}\sf{M}_p
\end{pmatrix}^{-1},
\end{equation*}
where 
\begin{equation*}
\sf{A}_u = 
\begin{pmatrix}
	\sf{A}_{DD} & \sf{A}_{DC} \\
	\sf{A}_{CD} & \sf{A}_{CC}  \\
\end{pmatrix},
\quad 
\sf{G} = 
\begin{pmatrix}
 \sf{G}_D \\
 \sf{G}_C \\
\end{pmatrix},
\end{equation*}
and $\sf{M}_p$ is the mass matrix, {\it i.e.},  $(p_h, q_h) \mapsto \sf{M}_p$. We want to point out that, in~\cite{YiEtAl22-Stokes}, the block corresponding to the velocity part is based on the mesh-dependent norm on $\Vh$.  Here, we directly use the blocks from the stiffness matrix $\mathcal{A}$. Lemma~\ref{lem: a} makes sure that this still leads to effective preconditioners.  As suggested~\cite{YiEtAl22-Stokes}, while the inverse of $\sf{M}_p$ is trivial since it is diagonal, inverting the diagonal block $\sf{A}_u$ could be expensive and sometimes infeasible. Therefore, we approximately invert this diagonal block and define the following inexact counterparts
\begin{equation*}
	\mathcal{M}_D =
	\begin{pmatrix}
		\sf{H}_{u} & \sf{0} \\
		\sf{0}    & \nu\sf{M}^{-1}_p
	\end{pmatrix},
	\quad 
	\mathcal{M}_L =
	\begin{pmatrix} 
		\sf{H}_{u}^{-1} & \sf{0} \\
		\sf{G}^{\top}   & \nu^{-1}\sf{M}_p
	\end{pmatrix}^{-1},
	\quad
	\mathcal{M}_U =
	\begin{pmatrix}
		\sf{H}_{u}^{-1} &  \sf{G} \\
		\sf{0}   & \nu^{-1}\sf{M}_p
	\end{pmatrix}^{-1},
\end{equation*}
where $\sf{H}_u$ is spectrally equivalent to $\sf{A}_u^{-1}$.  

Next, we consider the \texttt{PPR-EG} method, Algorithm~\ref{alg:EG-PR-Elimination}.  Since it is well-posed, as shown in Theorem~\ref{thm:perturb-eg-wellposed}, we can similarly develop the corresponding block preconditioners as follows:
\begin{equation*}
	\mathcal{B}^D_D =
	\begin{pmatrix}
		\sf{A}^D_{u} & \sf{0} \\
		\sf{0}    & \nu^{-1}\sf{M}_p
	\end{pmatrix}^{-1},
	\quad 
	\mathcal{B}^D_L =
	\begin{pmatrix} 
		\sf{A}^D_{u} & \sf{0} \\[5pt]
		\sf{G}^{\top}   & \nu^{-1}\sf{M}_p
	\end{pmatrix}^{-1},
	\quad
	\mathcal{B}^D_U =
	\begin{pmatrix}
		\sf{A}^D_{u} &  \sf{G} \\
		\sf{0}   & \nu^{-1}\sf{M}_p
	\end{pmatrix}^{-1},
\end{equation*}
where 
\begin{equation*}
	\sf{A}^D_u = 
	\begin{pmatrix}
		\sf{D}_{DD} & \sf{A}_{DC} \\
		\sf{A}_{CD} & \sf{A}_{CC}  \\
	\end{pmatrix},
\end{equation*}
and their inexact versions providing more practical values:
\begin{equation*}
	\mathcal{M}^D_D =
	\begin{pmatrix}
		\sf{H}^D_{u} & \sf{0} \\
		\sf{0}    & \nu\sf{M}^{-1}_p
	\end{pmatrix},
	\hskip .5pt
	\mathcal{M}^D_L =
	\begin{pmatrix} 
		(\sf{H}^D_{u})^{-1} & \sf{0}\\
		\sf{G}^{\top}   &
		{{\nu}^{-1}}\sf{M}_p
	\end{pmatrix}^{-1},
		\hskip .5pt
	\mathcal{M}^D_U =
	\begin{pmatrix}
		(\sf{H}^D_{u})^{-1} &  \sf{G} \\
		\sf{0}   & 
		{{\nu}^{-1}}\sf{M}_p
	\end{pmatrix}^{-1},
\end{equation*}
where $\sf{H}^D_u$ is spectrally equivalent to $(\sf{A}^D_u)^{-1}$.  

Finally, we consider the \texttt{CPR-EG} method, Algorithm~\ref{alg:EG-elimination-2}.
The following block preconditioners are constructed based on the well-posedness of the \texttt{CPR-EG} method: 
\begin{equation*}
	\mathcal{B}^E_D =
	\begin{pmatrix}
		\sf{A}^E_{u} & \sf{0} \\
		\sf{0}    &  \sf{S}_p^E
	\end{pmatrix}^{-1},
	\quad 
	\mathcal{B}^E_L =
	\begin{pmatrix} 
		\sf{A}^E_{u} & \sf{0} \\
		(\sf{G}^E)^{\top}   & \sf{S}^E_p
	\end{pmatrix}^{-1},
	\quad
	\mathcal{B}^E_U =
	\begin{pmatrix}
		\sf{A}^E_{u} &  \sf{G}^E \\
		\sf{0}   & \sf{S}^E_p
	\end{pmatrix}^{-1},
\end{equation*}
where $\sf{S}_p^E:=\nu^{-1}\sf{M}_p + \sf{A}^E_p = \nu^{-1}\sf{M}_p + \sf{G}^{\top}_{D}  \sf{D}_{DD}^{-1}\sf{G}_{D}$.  In this case, the second diagonal block, $\sf{S}_p^E$, is not a diagonal matrix anymore. Therefore, in the inexact version block preconditioners, we need to replace $(\sf{S}^E_p)^{-1}$ by its spectrally equivalent approximation $\sf{H}^E_p$.  This leads to the following inexact block preconditioners:
\begin{equation*}
	\mathcal{M}^E_D =
	\begin{pmatrix}
		\sf{H}^E_{u} & \sf{0} \\
		\sf{0}    & \sf{H}^{E}_p
	\end{pmatrix},
	\  
	\mathcal{M}^D_L =
	\begin{pmatrix} 
		(\sf{H}^E_{u})^{-1} & \sf{0}\\
		(\sf{G}^E)^{\top}   & (\sf{H}^E_p)^{-1}
	\end{pmatrix}^{-1},
	\ 
	\mathcal{M}^E_U =
	\begin{pmatrix}
		(\sf{H}^D_{u})^{-1} &  \sf{G}^E \\
		\sf{0}   & (\sf{H}^E_p)^{-1}
	\end{pmatrix}^{-1},
\end{equation*}
where $\sf{H}^E_u$ and $\sf{H}^E_p$ are spectrally equivalent to $(\sf{A}^E_u)^{-1}$ and $(\sf{A}^E_p)$, respectively.

Before closing this section, we want to point out that, following the general framework in \cite{loghin2004analysis, mardal2011preconditioning}, we can show that the  block diagonal preconditioners are parameter-robust and can be applied to the minimal residual methods.  On the other hand, following the framework presented in~\cite{loghin2004analysis,adler2017robust,adler2020robust}, we can also show that block triangular preconditioners are field-of-value equivalent preconditioners and can be applied to the {generalized} minimal residual (GMRES) method.  
We omit the proof here, but we refer the readers to our previous work~\cite{adler2017robust,adler2020robust,YiEtAl22-Stokes} for similar proofs.

\section{Numerical Examples}
\label{sec:nume_examples}
In this section, we conduct numerical experiments to validate our theoretical conclusions presented in the previous sections.
The numerical experiments are implemented by authors' codes developed based on iFEM \cite{CHE09}.
To distinguish the numerical solutions using the four different EG algorithms considered in the present work, we use the following notations: 
\begin{itemize}
    \item $(\bu_h^{\texttt{ST}},p_h^{\texttt{ST}})$: Solution by the \texttt{ST-EG} method in Algorithm~\ref{alg:EG}.
    \item $(\bu_h^{\texttt{PR}},p_h^{\texttt{PR}})$: Solution by the \texttt{PR-EG} method in Algorithm~\ref{alg:EG-PR}.
    \item $(\bu_h^{\texttt{PPR}},p_h^{\texttt{PPR}})$: Solution by the \texttt{PPR-EG} method in Algorithm~\ref{alg:EG-PR-Elimination}.
    \item $(\bu_h^{\texttt{CPR}},p_h^{\texttt{CPR}})$: Solution by the \texttt{CPR-EG} method in
    Algorithm~\ref{alg:EG-elimination-2}.
\end{itemize}

The error estimates for the \texttt{ST-EG} method have been proved in \cite{YiEtAl22-Stokes}:
\begin{subequations}\label{sys: errboundst}
\begin{alignat}{1}
& \|\bu-\bu_h^{\texttt{ST}}\|_\mathcal{E}\lesssim h \left( \|\mathbf{u}\|_2+ \nu^{-1}\|p\|_1 \right ), \label{eqn: errboundstu} \\
& \|\mathcal{P}_0p-p_h^{\texttt{ST}}\|_0\lesssim  h \left( \nu \|\mathbf{u}\|_2+ \|p\|_1 \right ), \quad 
\|p-p_h^{\texttt{ST}}\|_0 \lesssim  h \left ( \nu\|\mathbf{u}\|_2+ \|p\|_1\right ).\label{eqn: errboundstp}
\end{alignat}
\end{subequations}
Also, recall the error estimates for the 
\texttt{PR-EG} method:
\begin{subequations}\label{sys: errboundpr}
\begin{alignat}{1}
& \|\bu-\bu_h^{\texttt{PR}}\|_\mathcal{E} \lesssim h\|\mathbf{u}\|_2, \label{eqn: errboundpru} \\
& \|\mathcal{P}_0p-p_h^{\texttt{PR}}\|_0 \lesssim \nu h\|\mathbf{u}\|_2, \quad
\|p-p_h^{\texttt{PR}}\|_0 \lesssim  h \left ( \nu \|\mathbf{u}\|_2+ \|p\|_1 \right).  \label{eqn: errboundprp}
\end{alignat}
\end{subequations}
Note that $(\bu_h^{\texttt{PPR}},p_h^{\texttt{PPR}})$ and $(\bu_h^{\texttt{CPR}},p_h^{\texttt{CPR}})$ satisfy the same error estimates as in \eqref{sys: errboundpr}.

In two- and three-dimensional numerical examples, we show the optimal convergence rates  and pressure-robustness of the pressure-robust EG methods (\texttt{PR-EG}, \texttt{PPR-EG}, and \texttt{CPR-EG}) via mesh refinement study and by considering a wide range of the viscosity value. To highlight the better performance of the pressure-robust EG methods than the \texttt{ST-EG} method, we compare the magnitudes and behaviors of the errors produced by them as well. We also demonstrate the improved computational efficiency of the \texttt{CPR-EG} method compared to the \texttt{PR-EG} method. Further, we present some results of the performance of the block preconditioners developed in Subsection~\ref{sec:block-prec} on three-dimensional benchmark problems to show its robustness and effectiveness.

\subsection{A two-dimensional example}
We consider an example problem in two dimensions.
In this example, the penalty parameter was set to $\rho=10$. 

\subsubsection{Test 1: Vortex flow}\label{sect:NumTest-1}
Let the computational domain be $\Omega=(0,1)\times (0,1)$. The velocity field and pressure are chosen as
\begin{equation*}
    \bu
    = \left(\begin{array}{c}
    10x^2(x-1)^2y(y-1)(2y-1) \\
    -10x(x-1)(2x-1)y^2(y-1)^2
    \end{array}\right),
    \quad
    p = 10(2x-1)(2y-1).
\end{equation*}
Then the body force $\bbf$ and the Dirichlet boundary condition $\bu = \mathbf{g}$ are obtained from \eqref{sys: governing} using the exact solutions.

\noindent\textbf{Accuracy test.}
First, we perform a mesh refinement study for both the \texttt{ST-EG} and \texttt{PR-EG} methods by varying the mesh size $h$ while keeping  $\nu = 10^{-6}$.
The results are summarized in Table~\ref{test1_h}, where we observe that the  convergence rates for the velocity and pressure errors for both methods are of at least first-order.
The velocity error for the \texttt{ST-EG} method
seems to converge at the order of $1.5$, but the \texttt{PR-EG} method yields about five orders of magnitude smaller velocity errors than the standard method. On the other hand, the total pressure errors produced by the two methods are very similar in magnitude. 
Therefore, our numerical results support our theoretical error estimates in \eqref{sys: errboundst} and \eqref{sys: errboundpr}.

\setlength\tabcolsep{4.2pt}
\begin{table}[!h]
    \centering
    \begin{tabular}{|c||c|c|c|c||c|c|c|c|}
    \hline
        &  \multicolumn{4}{c||}{\texttt{ST-EG}
        } &  \multicolumn{4}{c|}{\texttt{PR-EG}
        }\\
    \cline{2-9}   
       $h$  & {\small $\enorm{\bu-\buh^{\texttt{ST}}}$} & {\small Rate} & {\small$\|p-p_h^{\texttt{ST}}\|_0$} & {\small Rate} & {\small$\|\bu-\bu_h^{\texttt{PR}}\|_{\mathcal{E}}$} & {\small Rate} & {\small$\|p-p_h^{\texttt{PR}}\|_0$} & {\small Rate} \\ 
       \hline
       $1/4$ & 1.959e+5 & - & 1.166e+0 & -  & 2.200e-1 & - & 9.547e-1 & - \\
       \hline
       $1/8$  & 7.140e+4 & 1.46 & 5.180e-1 & 1.17  & 1.060e-1 & 1.05 & 4.802e-1 & 0.99 \\
       \hline
       $1/16$  & 2.468e+4 & 1.53 & 2.483e-1 & 1.06  & 4.920e-2 & 1.11 & 2.404e-1 & 1.00 \\
       \hline
       $1/32$  & 8.552e+3 & 1.53 & 1.223e-1 & 1.02 & 2.372e-2 & 1.05 & 1.203e-1 & 1.00 \\
       \hline
       $1/64$  & 2.987e+3 & 1.52 & 6.078e-2 & 1.01 &  1.166e-2 & 1.02 & 6.014e-2 & 1.00\\
       \hline
    \end{tabular}
    \caption{Test~\ref{sect:NumTest-1}, Vortex flow: A mesh refinement study for the \texttt{ST-EG} and \texttt{PR-EG} methods with varying mesh size $h$ and a fixed viscosity $\nu = 10^{-6}$.}
    \label{test1_h}
\end{table}

Next, to visualize the quality difference in the solutions, we present the numerical solutions obtained by the two methods with $h=1/16$ and $\nu=10^{-6}$ in Figure~\ref{test1_numesol}. 
As expected, the two methods produce nearly the same pressure solutions. As for the velocity solutions, the \texttt{PR-EG} method well captures the vortex flow pattern, while the \texttt{ST-EG} method is unable to do so. 

\begin{figure}[!htb]
\begin{subfigure}{\linewidth}
\centering
    \includegraphics[width=3.5cm]{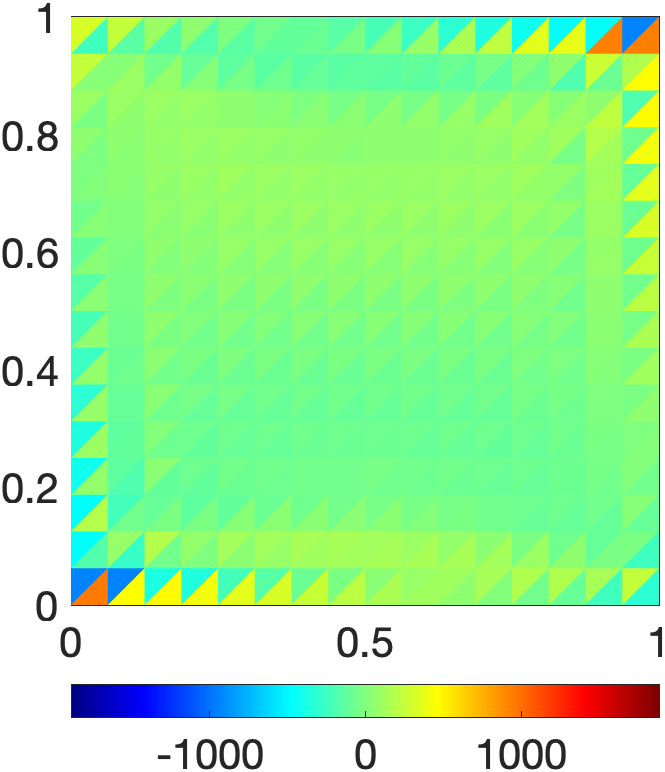}
    \hskip 15pt
    \includegraphics[width=3.5cm]{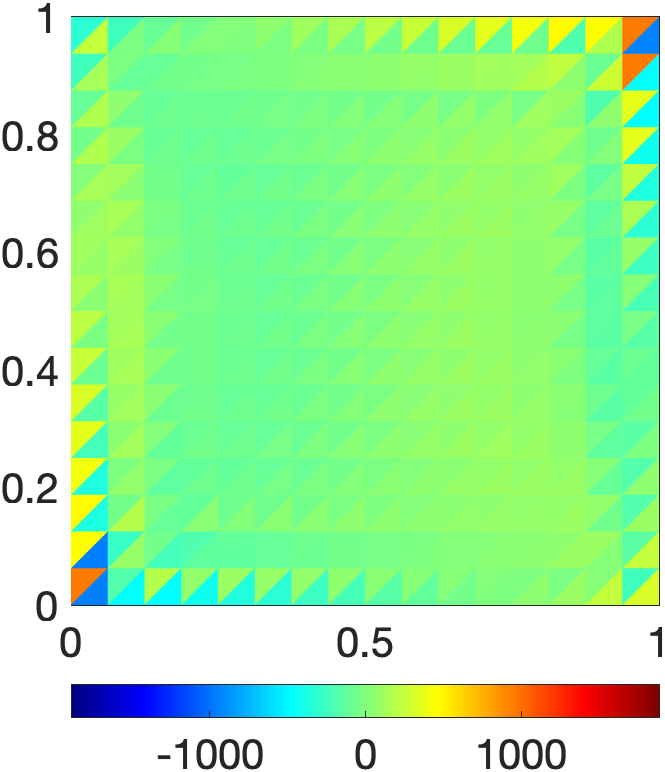}
    \hskip 15pt
    \includegraphics[width=3.5cm]{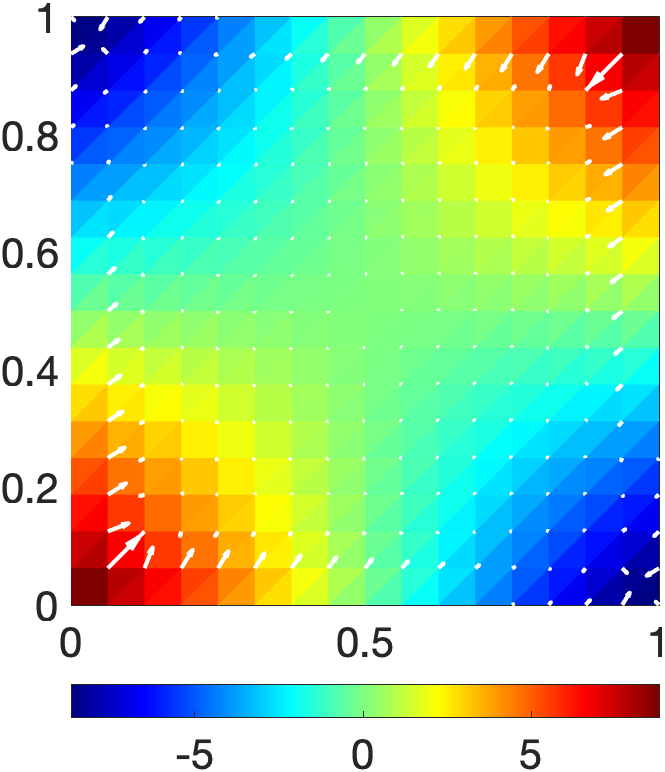}
    \caption{\texttt{ST-EG}: $u_1$, $u_2$, and $p$ with the velocity vector fields, from left to right}
\end{subfigure}
\begin{subfigure}{\linewidth}
    \centering
    \includegraphics[width=3.5cm]{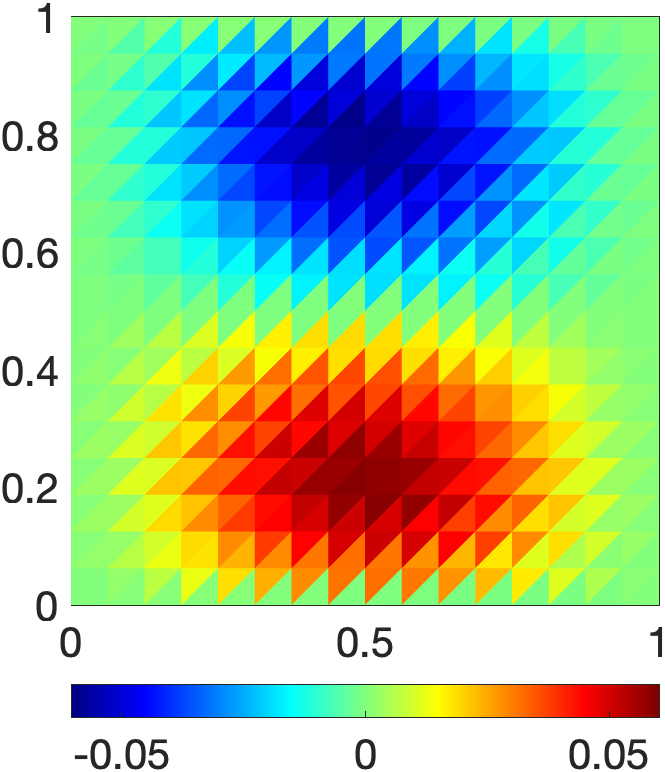}
    \hskip 15pt
    \includegraphics[width=3.5cm]{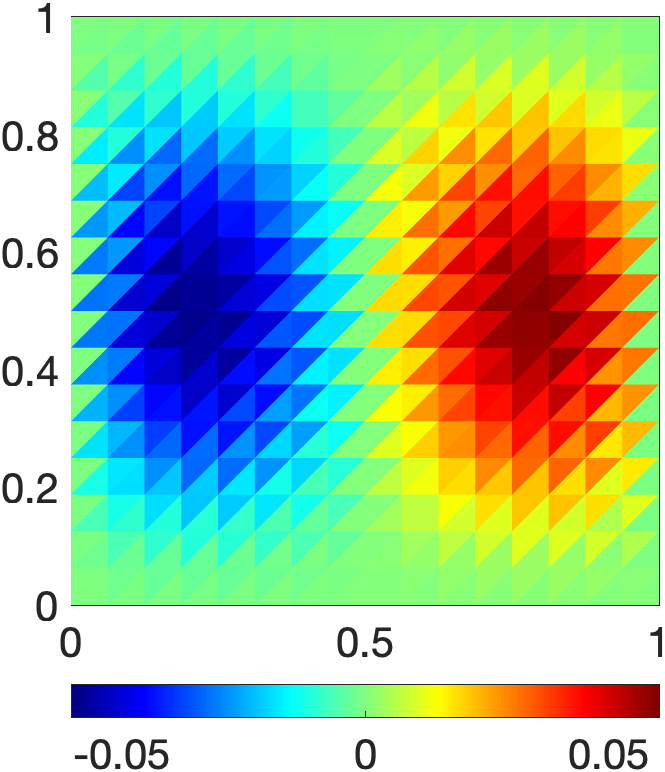}
    \hskip 15pt
    \includegraphics[width=3.5cm]{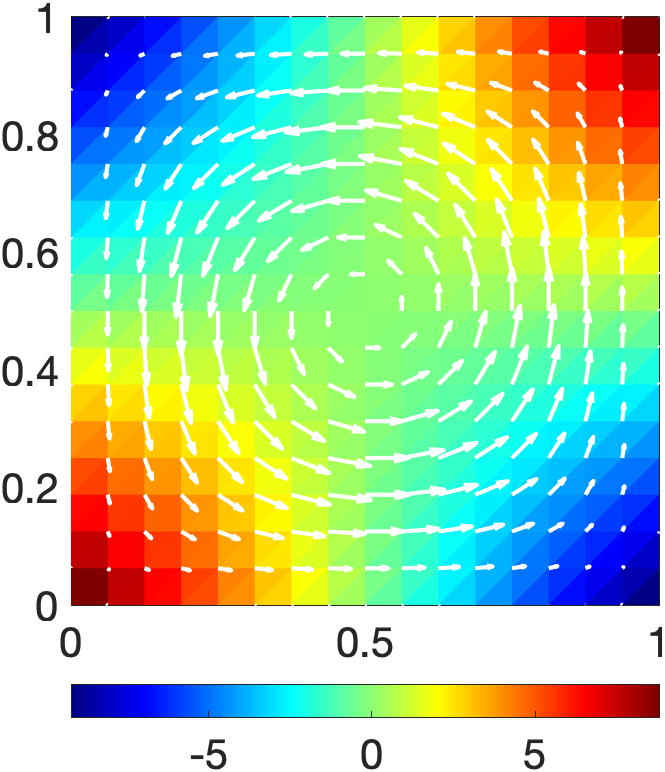}
    \caption{\texttt{PR-EG}: $u_1$, $u_2$, and $p$ with the velocity vector fields, from left to right}
\end{subfigure}
    \caption{Test~\ref{sect:NumTest-1}, Vortex flow: Comparison of the numerical solutions with $h=1/16$ and $\nu=10^{-6}$.}
    \label{test1_numesol}
\end{figure}

\noindent\textbf{Robustness test.}
This test is to verify the pressure-robustness of the \texttt{PR-EG} method. To confirm the error behaviors predicted by \eqref{sys: errboundst} and \eqref{sys: errboundpr}, we solved the example problem with varying $\nu$ values,  from $10^{-2}$ to $10^{-6}$, while fixing the mesh size to $h = 1/32$. 
Figure~\ref{test1_figure_nu} shows the total velocity and auxiliary pressure errors, {\it i.e.,} $\enorm{\bu-\bu_h}$ and $\norm{\mathcal{P}_0p-p_h}_0$.
As expected, the \texttt{ST-EG} method produces the velocity errors inversely proportional to $\nu$ as the second term in the error bound \eqref{eqn: errboundstu} becomes a dominant one as $\nu$ gets smaller. Meanwhile, the auxiliary pressure error remains nearly the same while $\nu$ varies.
On the other hand, the \texttt{PR-EG} method produces nearly the same velocity errors regardless of the $\nu$ values.
Moreover, the auxiliary pressure errors decrease in proportion to $\nu$.
These numerical results are consistent with the error bounds \eqref{sys: errboundst} and \eqref{sys: errboundpr}.
\begin{figure}[!h]
\centering
\begin{subfigure}{0.45\linewidth}
\centering
    \includegraphics[width=5.5cm]{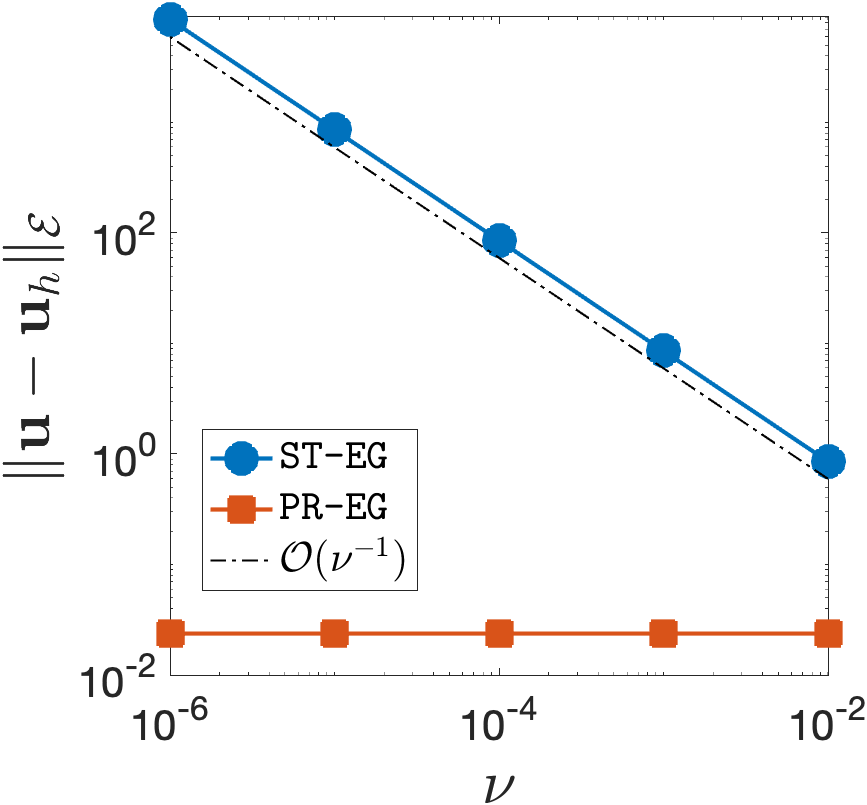}
    \caption{Velocity error vs. viscosity}
    \end{subfigure}
 \begin{subfigure}{0.45\linewidth}
\centering   
    \includegraphics[width=5.5cm]{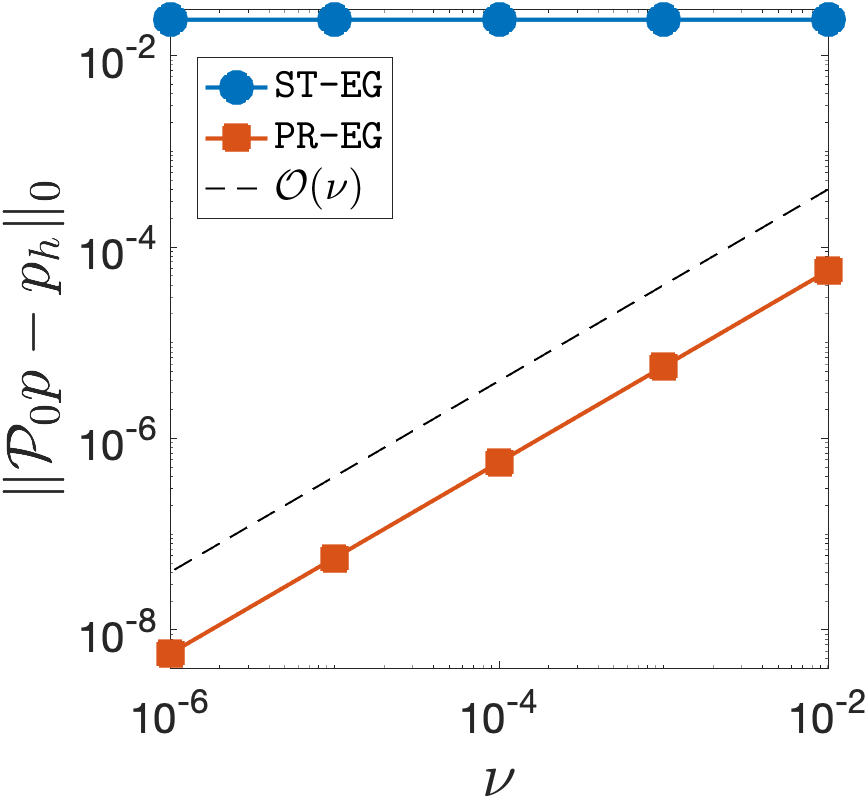}
\caption{Auxiliary pressure error vs. viscosity}
\end{subfigure}
    \caption{Test~\ref{sect:NumTest-1}, Vortex flow: Error profiles of the \texttt{ST-EG} and \texttt{PR-EG} methods with varying $\nu$ values and a fixed mesh size $h=1/32$.}
    \label{test1_figure_nu}
\end{figure}

\noindent\textbf{Performance of the \texttt{CPR-EG} method.}
In this test, we shall validate the error estimates in Theorem~\ref{thm:perturb_EG_error} for the \texttt{CPR-EG} method and compare its performances with the \texttt{PR-EG} method. 
First, we consider the sparsity pattern of their stiffness matrices when generated on the same mesh of size $h = 1/32$.
Figure~\ref{fig:Test1-sparsity} compares the sparsity patterns of the stiffness  matrices corresponding to \texttt{PR-EG}, \texttt{PPR-EG}, and \texttt{CPR-EG} methods. The \texttt{PPR-EG} method produces a matrix with less nonzero entries than the \texttt{PR-EG} method. It also shows the \texttt{CPR-EG} method
yields a much smaller but denser stiffness matrix  than the \texttt{PR-EG} method.
More specifically, by eliminating the DG component of the velocity vector, we can achieve 33\% reduction in the number of DoFs. To compare the accuracy of the two methods, we performed a numerical convergence study with varying $h$ values and a fixed viscosity $\nu = 10^{-6}$, whose results are plotted in Figure~\ref{fig:Test1-EGwithElimination}. As observed in this figure, the errors produced by the \texttt{CPR-EG} method
not only decrease at the optimal order of $\mathcal{O}(h)$, they are also nearly identical to those produced by the \texttt{PR-EG} method.
\begin{figure}[!htb]
\centering
    \begin{subfigure}{0.3\linewidth}
        \centering
\includegraphics[width=\textwidth]{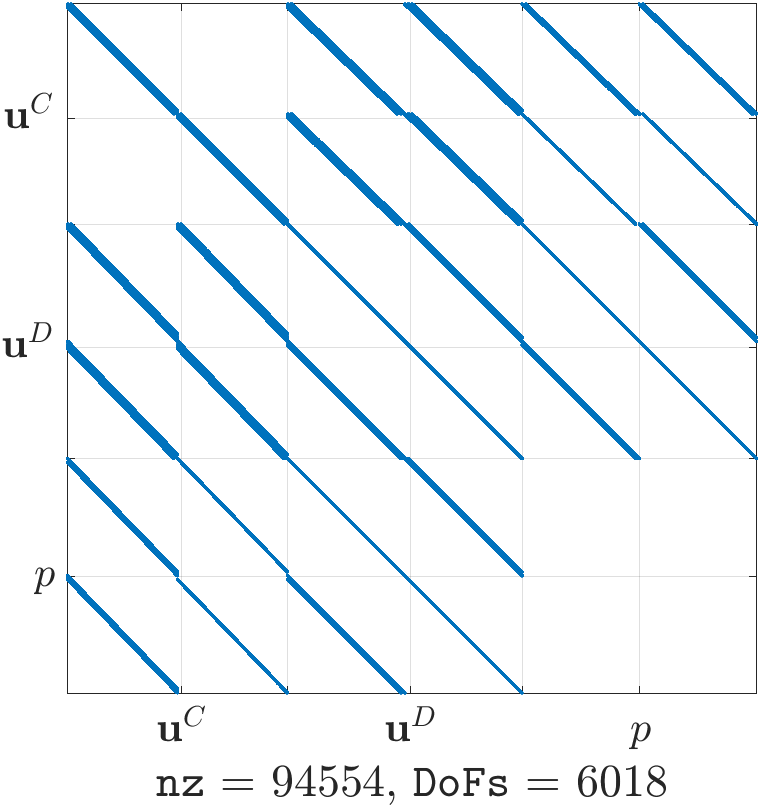}
\caption{\texttt{PR-EG}}
\end{subfigure}
\hskip 10pt
    \begin{subfigure}{0.3\linewidth}
        \centering
\includegraphics[width=\textwidth]{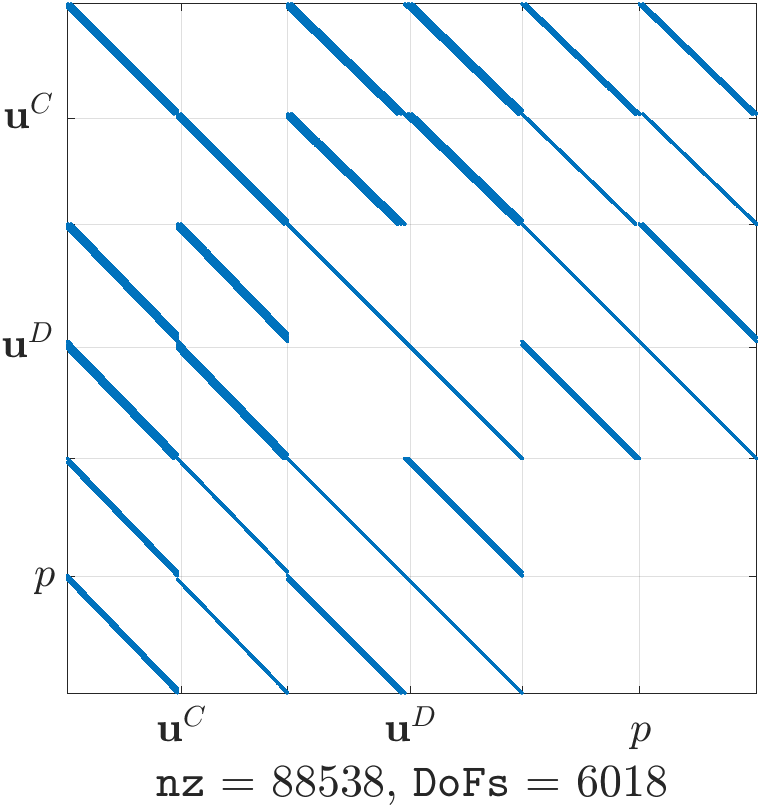}
\caption{\texttt{PPR-EG}}
\end{subfigure}
\hskip 10pt
\begin{subfigure}{0.3\linewidth}
        \centering
        \includegraphics[width=\textwidth]{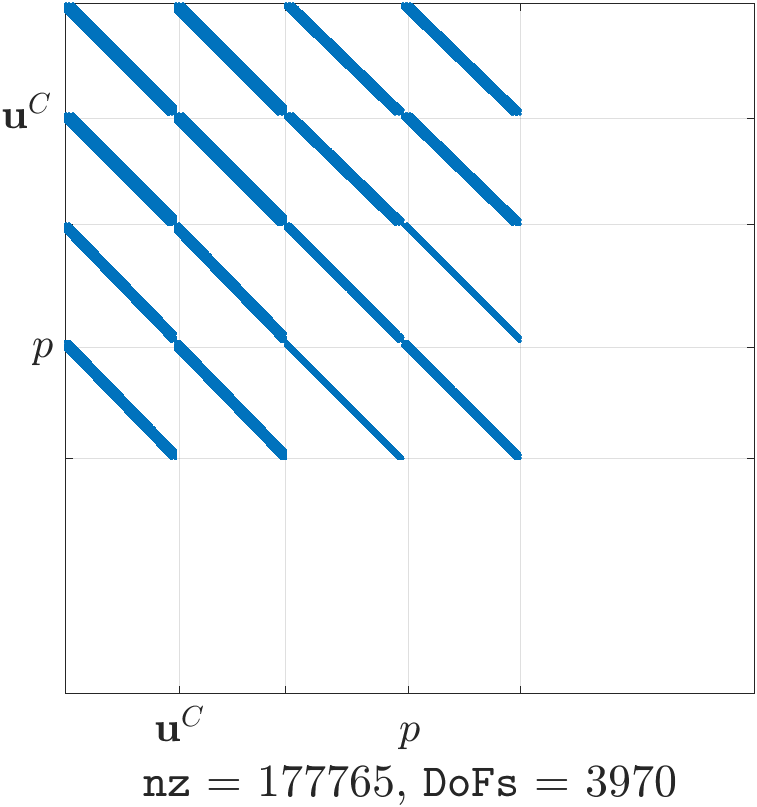}\\
\caption{\texttt{CPR-EG}}
\end{subfigure}
\caption{Test~\ref{sect:NumTest-1}, Vortex flow: Comparison of the sparsity patterns of the stiffness matrices on a mesh with $h = 1/32$. \texttt{nz} denotes the number of nonzeros.}\label{fig:Test1-sparsity}
\end{figure}

\begin{figure}[!h]
\centering
\includegraphics[width=.35\textwidth]{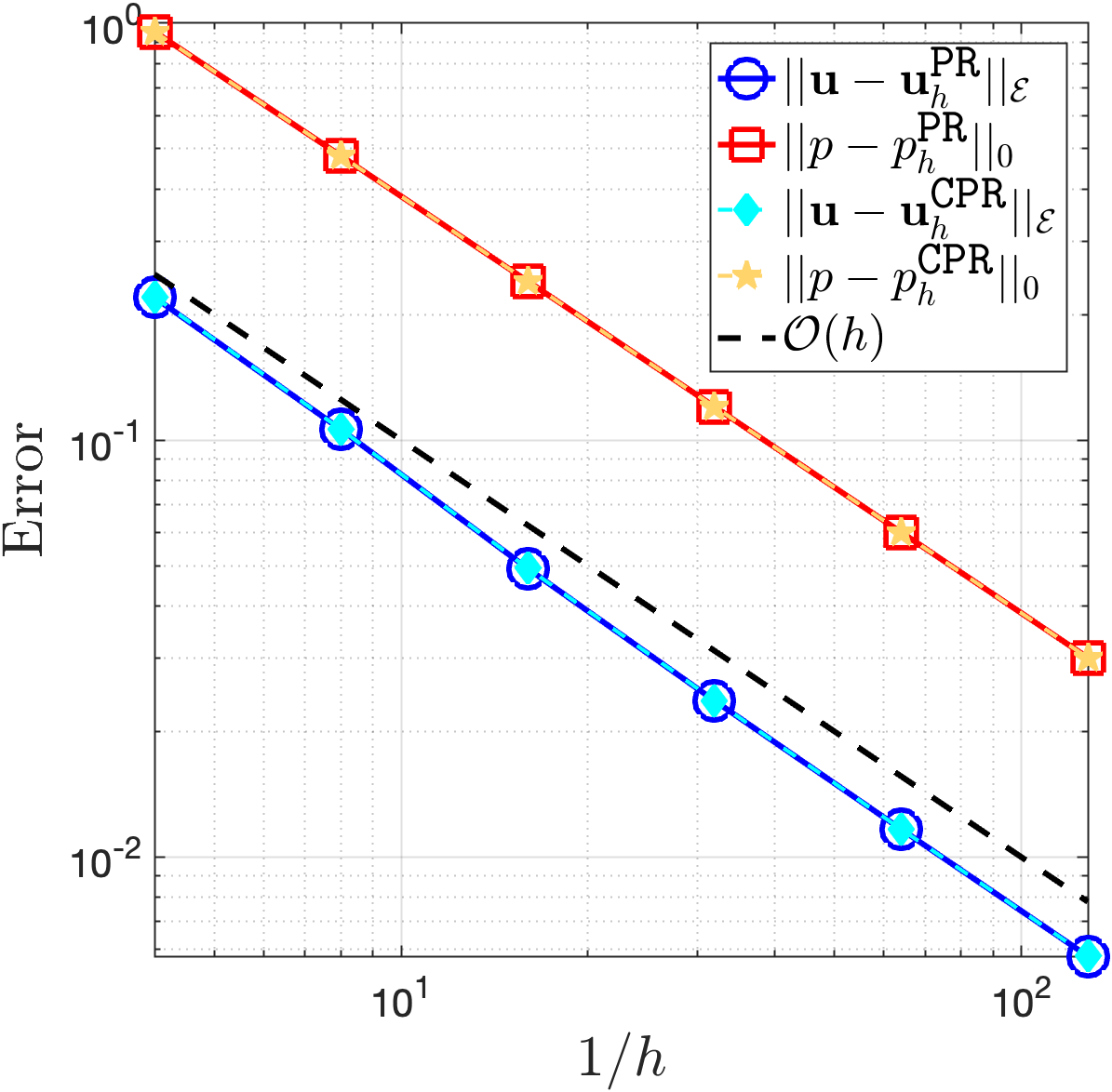}
\caption{Test~\ref{sect:NumTest-1}, Vortex flow: Comparison of the errors from the \texttt{PR-EG} and \texttt{CPR-EG} methods. The viscosity is fixed to $\nu = 10^{-6}$. 
}\label{fig:Test1-EGwithElimination}
\end{figure}

\subsection{Three dimensional examples}
We now turn our attention to some numerical examples in three dimensions. In {all three-dimensional tests}, we used the penalty parameter {$\rho = 2$}.

\subsubsection{Test 2: 3D flow in a unit cube}
\label{sect:NumTest-5}
In this example, we consider a 3D flow in a unit cube $\Omega=(0,1)^3$. The velocity field and pressure are chosen as
\begin{equation*}
    \bu 
    = \left(\begin{array}{c}
    \sin(\pi x)\cos(\pi y) - \sin(\pi x)\cos(\pi z) \\
    \sin(\pi y)\cos(\pi z) - \sin(\pi y)\cos(\pi x) \\
    \sin(\pi z)\cos(\pi x) - \sin(\pi z)\cos(\pi y)
    \end{array}\right),\quad
    p = \sin(\pi x)\sin(\pi y)\sin(\pi z).
\end{equation*}

\noindent\textbf{Accuracy test.}
With this example problem, we performed a mesh refinement study for the \texttt{ST-EG} and \texttt{PR-EG} methods with a fixed viscosity $\nu = 10^{-6}$. The results, summarized in Table~\ref{test5_h}, show very similar convergence behaviors to those of the two-dimensional results.
Though the velocity errors generated by the \texttt{ST-EG} method appear to decrease super-linearly, 
the magnitudes of the numerical velocity solutions are extremely larger than those of the exact solution and the solution of the \texttt{PR-EG} method. See the streamlines of the velocity solutions of the \texttt{ST-EG} and \texttt{PR-EG} methods in Figure~\ref{test5_numesol}.

\setlength\tabcolsep{4.2pt}
\begin{table}[!h]
    \centering
    \begin{tabular}{|c||c|c|c|c||c|c|c|c|}
    \hline
        &  \multicolumn{4}{c||}{\texttt{ST-EG}
        } &  \multicolumn{4}{c|}{ \texttt{PR-EG}
        }\\
    \cline{2-9}   
       $h$   & {\small $\enorm{\bu-\buh^{\texttt{ST}}}$} & {\small Rate} & {\small$\|p-p_h^{\texttt{ST}}\|_0$} & {\small Rate} & {\small$\|\bu-\bu_h^{\texttt{PR}}\|_{\mathcal{E}}$} & {\small Rate} & {\small$\|p-p_h^{\texttt{PR}}\|_0$} & {\small Rate} \\
       \hline
      $1/4$  & 8.785e+3 & - & 1.058e-1 & - & 3.732e+0 & - & 9.581e-2 & - \\ 
       \hline
       $1/8$   & 3.429e+3 & 1.36 & 5.144e-2 & 1.04 & 1.827e+0 & 1.03 & 4.879e-2 & 0.97 \\
       \hline
       $1/16$  & 1.239e+3 & 1.47 & 2.514e-2 & 1.03 & 9.048e-1 & 1.01 & 2.451e-2 & 0.99 \\
       \hline
       $1/32$  & 4.346e+2 & 1.51 & 1.241e-2 & 1.02 & 4.501e-1 & 1.01 & 1.227e-2 & 1.00 \\
       \hline
       $1/64$   & 1.521e+2 & 1.51 & 6.171e-3 & 1.01 &  2.244e-1 & 1.00 & 6.135e-3 & 1.00\\
       \hline
    \end{tabular}
    \caption{Test~\ref{sect:NumTest-5}, 3D flow in a unit cube: A mesh refinement study for the \texttt{ST-EG} and \texttt{PR-EG} methods on uniform meshes with varying $h$ and a fixed viscosity $\nu = 10^{-6}$.}
    \label{test5_h}
\end{table}

\begin{figure}[!htb]
\centering
\begin{subfigure}{.45\linewidth}
\centering
    \includegraphics[width=.75\textwidth]{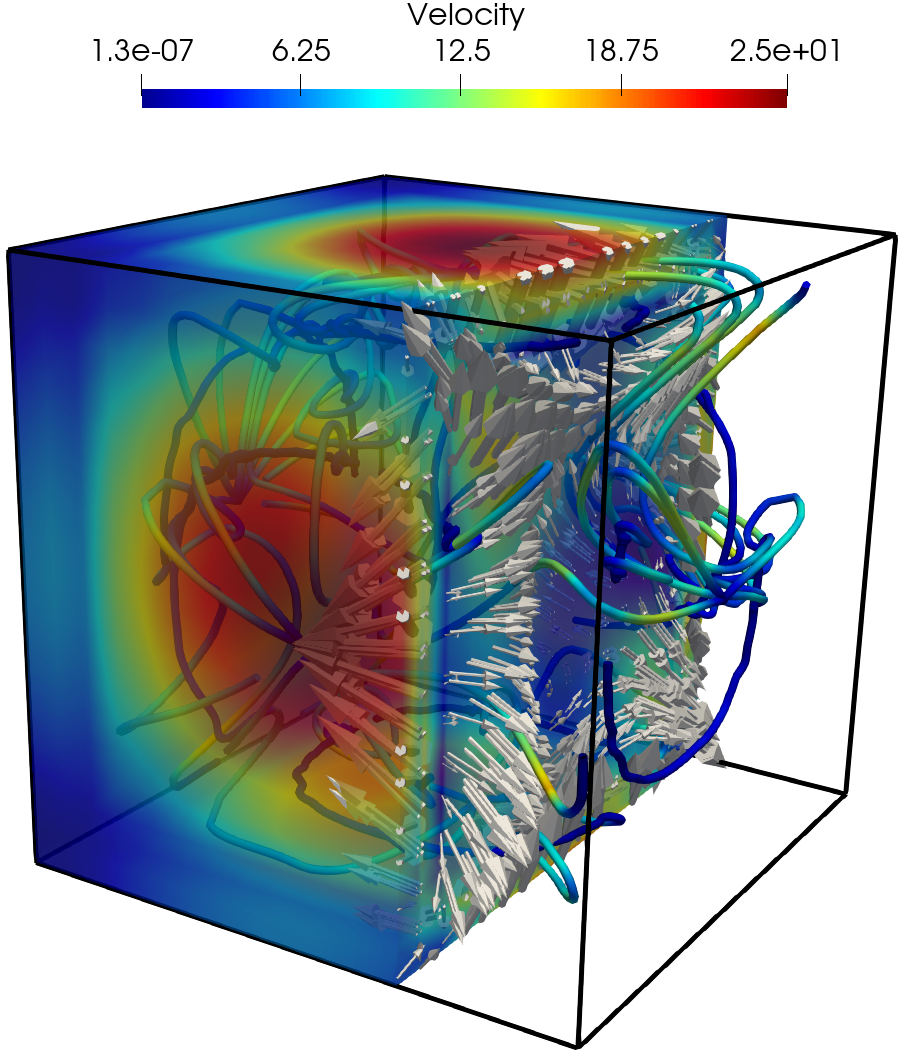}
    \caption{\texttt{ST-EG}}
    \hskip 20pt
    \end{subfigure}
    \begin{subfigure}{.45\linewidth}
   \centering
   \includegraphics[width=.7\textwidth]{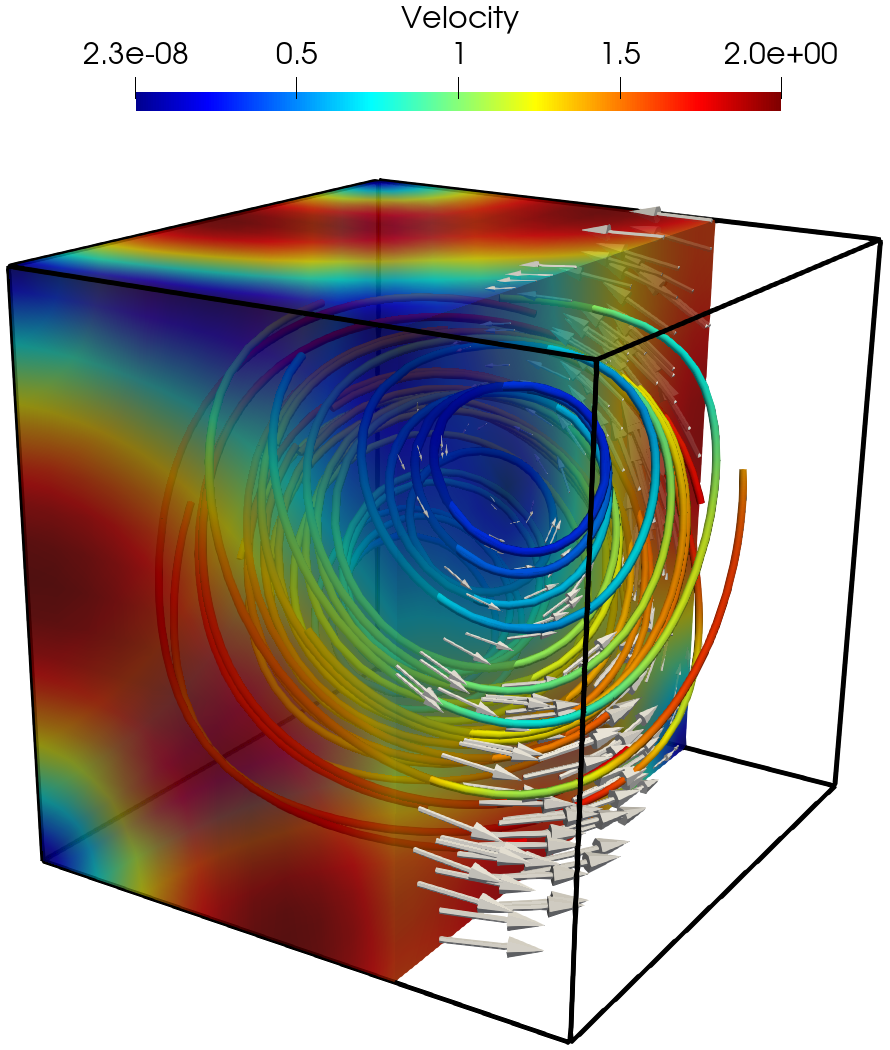}
   \caption{\texttt{PR-EG}}
   \end{subfigure}
    \caption{Test~\ref{sect:NumTest-5}, 3D flow in a unit cube: Streamlines of the numerical velocity when $h=1/16$ and $\nu=10^{-6}$.}
    \label{test5_numesol}
\end{figure}

\noindent\textbf{Robustness test.}
We consider the pattern of the error behaviors obtained by the two EG methods when $\nu$ varies and the mesh size is fixed to $h = 1/16$. The results of our tests are illustrated in Figure~\ref{test5_errors}.
As in the two-dimensional example, the velocity and (auxiliary) pressure errors for the \texttt{PR-EG} method follow the patterns predicted by the error estimates in \eqref{sys: errboundpr}. 
On the other hand, the velocity and auxiliary pressure errors for the  \texttt{ST-EG} method behave like  $\mathcal{O}(\nu^{-1})$ and $\mathcal{O}(1)$ only after $\nu$ becomes sufficiently small ($\nu \leq 10^{-3}$). The earlier deviation from these patterns when $\nu$ is relatively large ($10^{-3} \leq \nu \leq 10^{-2}$) is due to the smallness of $\norm{p}_1$ compared to $\norm{\bu}_2$, unlike in Test 1.
\begin{figure}[!htb]
\centering
\begin{subfigure}{0.45\linewidth}
    \centering
    \includegraphics[width=5.5cm]{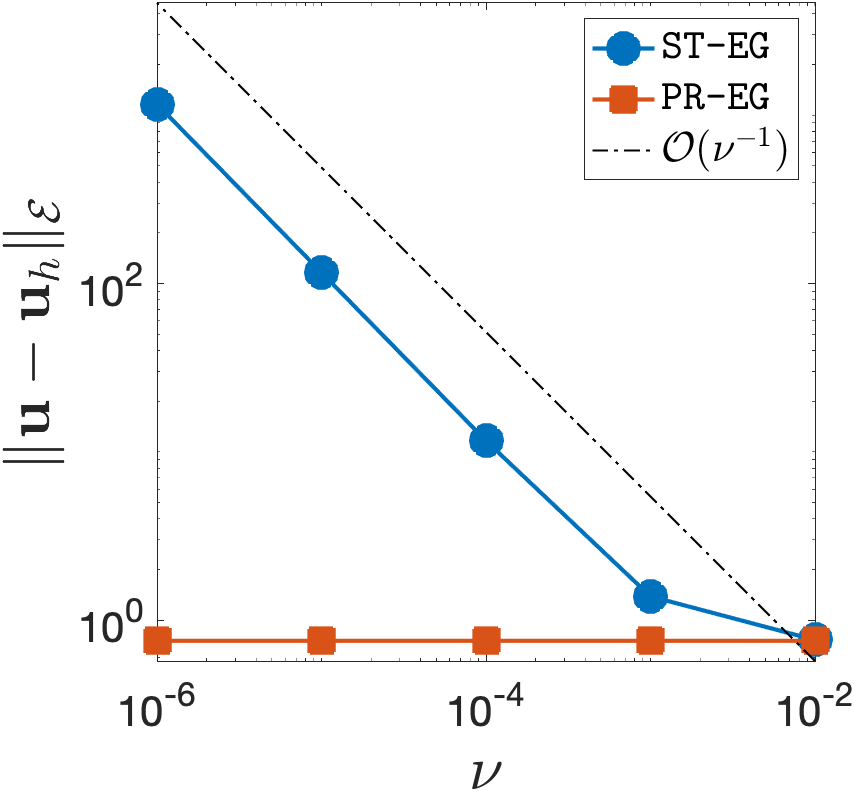}
    \caption{Velocity error vs. viscosity}
    \end{subfigure}
    \begin{subfigure}{0.45\linewidth}
    \centering
    \includegraphics[width=5.5cm]{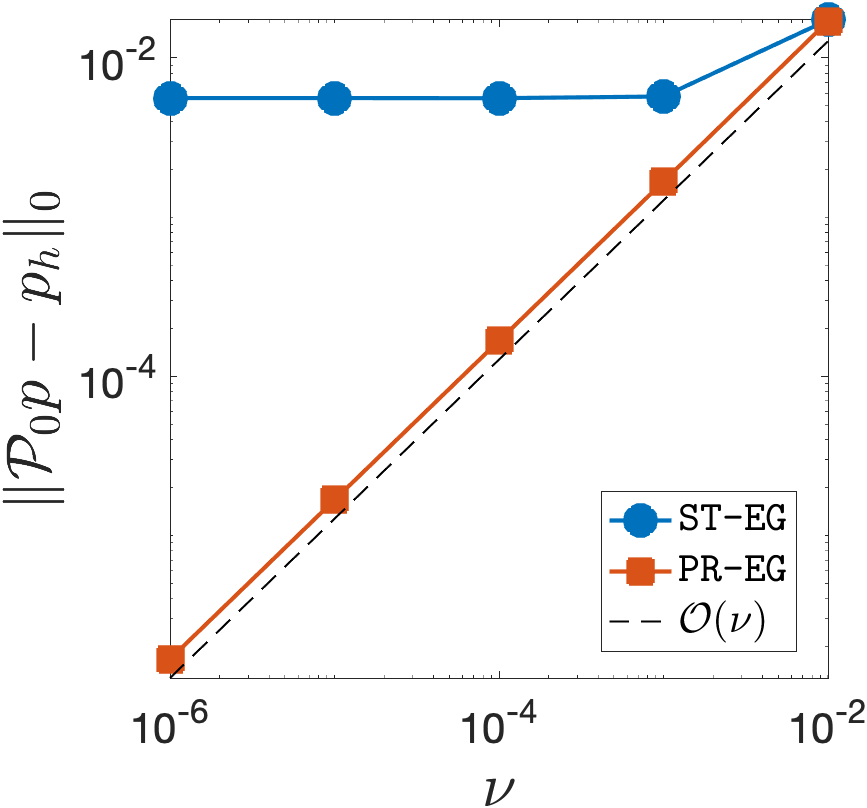}
        \caption{Auxiliary pressure error vs. viscosity}
    \end{subfigure}
    \caption{Test~\ref{sect:NumTest-5}, 3D flow in a unit cube: Error profiles of the \texttt{ST-EG} and \texttt{PR-EG} methods with varying $\nu$ and a fixed mesh size $h=1/16$.}
    \label{test5_errors}
\end{figure}

\noindent\textbf{Performance of the \texttt{CPR-EG} method.} 
Next, we shall demonstrate the savings in the computational cost when we use the \texttt{CPR-EG} method.
See Figure~\ref{fig:Test6-sparsity} for the sparsity patterns for the stiffness matrices generated by the \texttt{PR-EG}, \texttt{PPR-EG}, and \texttt{CPR-EG} methods on the same mesh with $h = 1/16$. 
In this case, the \texttt{PPR-EG} method generates a stiffness matrix with fewer nonzero entries compared to the \texttt{PR-EG} method.
On the other hand, the \texttt{CPR-EG} method requires approximately 38\% fewer DoFs than the other two methods. However, its resulting stiffness matrix is denser than those resulting from the other two methods.
Besides, we also compared the errors generated by the \texttt{PR-EG} and \texttt{PPR-EG} methods. They are nearly the same. But, the numerical data is not provided here for the sake of brevity. 

\begin{figure}[!htb]
\centering
\begin{subfigure}{0.3\linewidth}
\centering
\includegraphics[width=\textwidth]{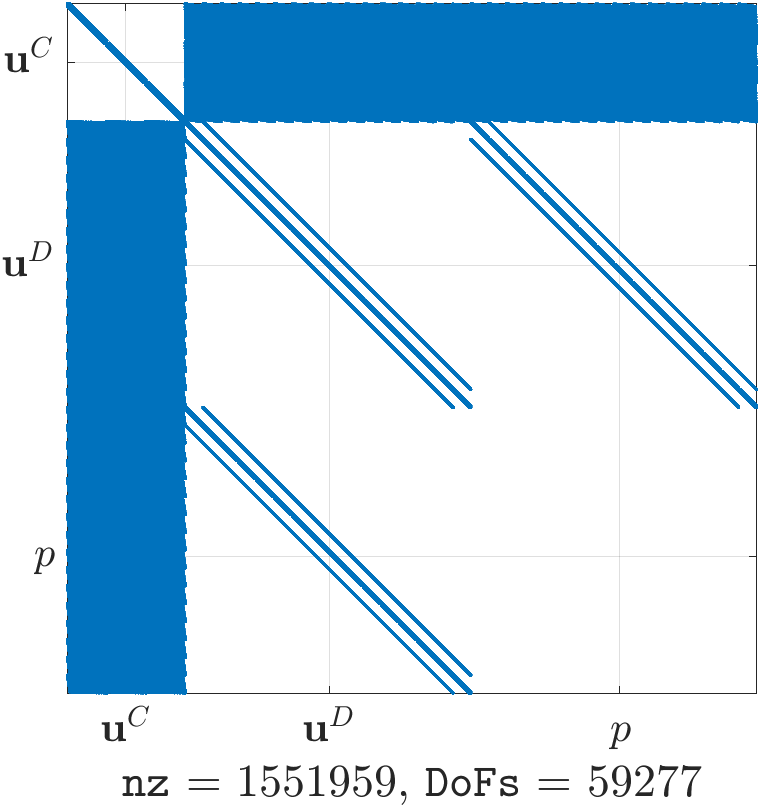}
\caption{\texttt{PR-EG}}
\end{subfigure}
\hskip 10pt
\begin{subfigure}{0.3\linewidth}
\centering
\includegraphics[width=\textwidth]{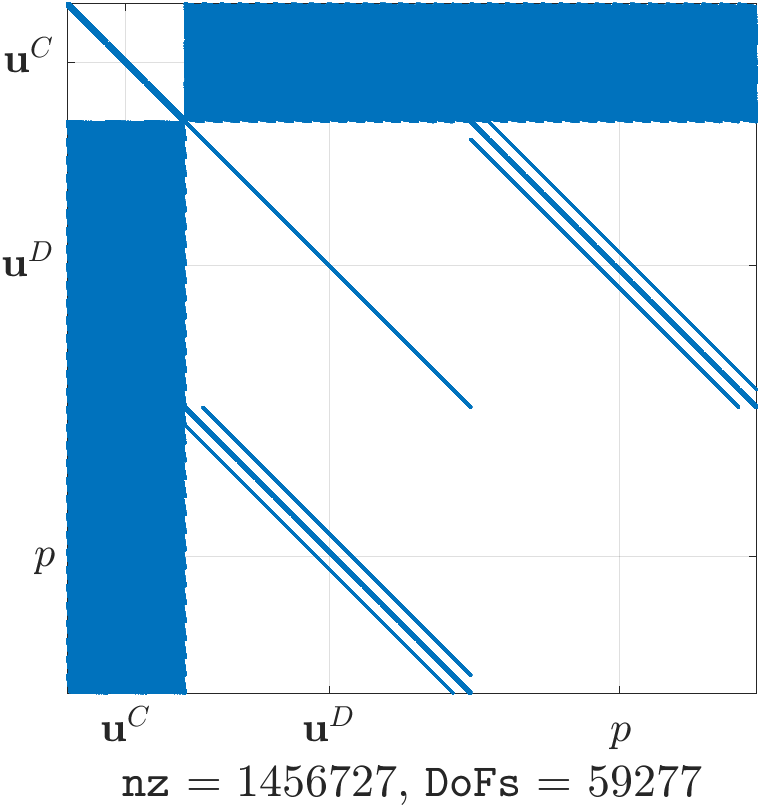}
\caption{\texttt{PPR-EG}}
\end{subfigure}
\hskip 10pt
\begin{subfigure}{0.3\linewidth}
\centering
\includegraphics[width=\textwidth]{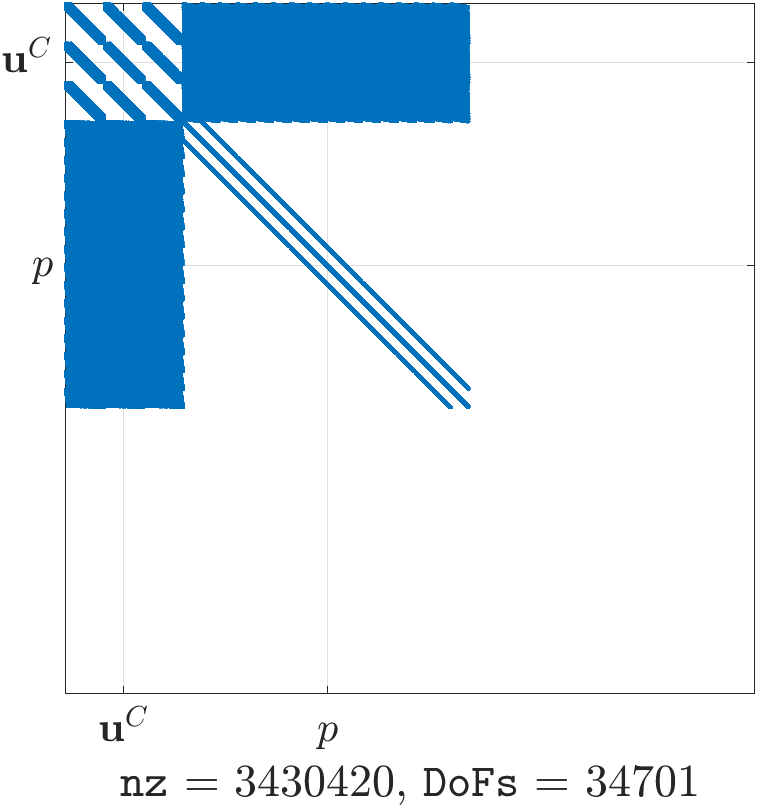}
\caption{\texttt{CPR-EG}}
\end{subfigure}
\caption{Test~\ref{sect:NumTest-5}, 3D flow in a unit cube: Comparison of the sparsity patterns of the stiffness matrices on a mesh with $h = 1/16$.}\label{fig:Test6-sparsity}
\end{figure}

\noindent\textbf{Performance of block preconditioners.}
We use the proposed block preconditioners to solve the corresponding linear systems and show their robustness with respect to the viscosity $\nu$. The required iteration numbers are reported in Table~\ref{test3_iter_nu} for mesh size $h = 1/4$ and penalty parameter $\rho = 2$. The exact and inexact block preconditioners are applied to the preconditioned GMRES method. For the inexact block preconditioners, we use an algebraic multigrid (AMG) preconditioned GMRES method to approximately invert the diagonal block with tolerance $10^{-6}$. This inner block solver usually took $5$-$7$ iterations in all our experiments. Therefore, we only report the outer GMRES iteration numbers. The results in Table~\ref{test3_iter_nu} show the robustness of our block preconditioners with respect to the viscosity $\nu$. This is further confirmed in Table~\ref{test3_kappa_nu}, where the condition numbers of the preconditioned stiffness matrices are reported. In Table~\ref{test3_kappa_nu}, we only show the results of {the application of the} block diagonal preconditioner since, in this case, the preconditioner is symmetric positive definite and the condition number is defined via the eigenvalues, i.e.,
$\kappa(\mathcal{B}_D \mathcal{A}) = \frac{\max |\lambda(\mathcal{B}_D \mathcal{A})| }{ \min |\lambda(\mathcal{B}_D \mathcal{A})|}$,  $\kappa(\mathcal{B}^D_D \mathcal{A}^D) = \frac{\max |\lambda(\mathcal{B}^D_D \mathcal{A}^D)| }{ \min |\lambda(\mathcal{B}^D_D \mathcal{A}^D)|}$, and $\kappa(\mathcal{B}^E_D \mathcal{A}^E) = \frac{\max |\lambda(\mathcal{B}^E_D \mathcal{A}^E)| }{ \min |\lambda(\mathcal{B}^E_D \mathcal{A}^E)|}$. 
As we can see, the condition number remains the same as $\nu$ decreases, which demonstrates the robustness of the proposed block preconditioners. The small variation in the number of iterations is mainly due to the outer GMRES method since the matrices $\mathcal{A}$, $\mathcal{A}^D$, and $\mathcal{A}^E$ are ill-conditioned and may affect the orthogonalization procedure used in Krylov iterative methods. In addition, the numerical performance for the \texttt{CPR-EG} method is the best among all three pressure-robust methods for this test in terms of the number of iterations. 

\begin{table}[H]
	\centering
	\begin{tabular}{|c||c|c|c||c|c|c||c|c|c|}
		\hline
		& \multicolumn{9}{c|}{Exact Solver} \\ \hline
		& \multicolumn{3}{c||}{
			\texttt{PR-EG}}&
		\multicolumn{3}{c||}{
			\texttt{PPR-EG}
		}
		& \multicolumn{3}{c|}{\texttt{CPR-EG}}\\
		\cline{2-10}
		$\nu$ & $\mathcal{B}_D$ & $\mathcal{B}_L$ & $\mathcal{B}_U$ 
		& $\mathcal{B}^D_D$ & $\mathcal{B}^D_L$ & $\mathcal{B}^D_U$ 
		& $\mathcal{B}^E_D$ & $\mathcal{B}^E_L$ & $\mathcal{B}^E_U$ \\ \hline
		1  &43  &23  &21    &62  &34  &32    &30  &20 &18\\ \hline
		$10^{-2}$ &61  &33  &33    &87  &49  &49    &45  &27 &28\\ \hline
		$10^{-4}$ &71  &39  &39    &89  &52  &52    &39  &25 &25\\ \hline
		$10^{-6}$ &72  &40  &40    &91  &55  &55    &36  &25 &25\\ \hline
		& \multicolumn{9}{c|}{Inexact Solver} \\ \hline
		& \multicolumn{3}{c||}{
			\texttt{PR-EG}
		}&
		\multicolumn{3}{c||}{
			\texttt{PPR-EG}}
		& \multicolumn{3}{c|}{\texttt{CPR-EG}}
		\\
		\cline{2-10}
		$\nu$ & $\mathcal{M}_D$ & $\mathcal{M}_L$ & $\mathcal{M}_U$
		& $\mathcal{M}^D_D$ & $\mathcal{M}^D_L$ & $\mathcal{M}^D_U$
		& $\mathcal{M}^E_D$ & $\mathcal{M}^E_L$ & $\mathcal{M}^E_U$ \\ \hline
		1  &43 &27 &25    &63 &37 &34   &34 &21 &19 \\ \hline
		$10^{-2}$  &61 &36 &35    &93 &56 &56   &53 &30 &31 \\ \hline
		$10^{-4}$  &75 &45 &45 	 &96 &61 &61   &47 &29 &28 \\ \hline
		$10^{-6}$  &83 &47 &47 	 &111&64 &64   &40 &28 &27\\ \hline
	\end{tabular}
	\caption{Test~\ref{sect:NumTest-5}, 3D flow in a unit cube:  Iteration counts for the block preconditioners when $\nu$ varies on mesh with $h = 1/4$.}
	\label{test3_iter_nu}
\end{table}

\begin{table}[H]
	\centering
	\begin{tabular}{|c||c||c||c|}
		\hline
		& \multicolumn{1}{c||}{
			\texttt{PR-EG}
		}&
		\multicolumn{1}{c||}{
			\texttt{PPR-EG}
		}
		& \multicolumn{1}{c|}{\texttt{CPR-EG}}\\
		\cline{2-4}
		$\nu$ & $\kappa(\mathcal{B}_D\mathcal{A})$ 
		& $ \kappa(\mathcal{B}^D_D \mathcal{A}^D)$ 
		& $ \kappa(\mathcal{B}^E_D \mathcal{A}^E)$ \\ \hline
		1  & 41.267  & 99.563 &  62.445 \\ \hline
		$10^{-2}$ & 41.267  & 99.563 &  62.445\\ \hline
		$10^{-4}$ & 41.267  & 99.563 &  62.445\\ \hline
		$10^{-6}$ & 41.267  & 99.563 &  62.445 \\ \hline
	\end{tabular}
	\caption{Test~\ref{sect:NumTest-5}, 3D flow in a unit cube: Condition number of 
	the preconditioned stiffness matrices  with the exact block preconditioners when $\nu$ varies on mesh with $h = 1/4$.}
	\label{test3_kappa_nu}
\end{table}

\subsubsection{Test 3: 3D vortex flow in L-shaped cylinder}
\label{sect:NumTest-6}
Let us consider an L-shaped cylinder defined by $\Omega=(0,1)^3\setminus (0.5,1)\times(0.5,1)\times(0,1)$.
In this domain, the exact velocity field and pressure are chosen as
\begin{equation*}
    \bu 
    = \frac{1}{x^2+y^2+1}\left(\begin{array}{c}
    -y \\
    x \\
    0
    \end{array}\right),\quad
    p = |2x-1|.
\end{equation*}
Note that the velocity is a rotational vector field whose center is $(x,y) = (0,0)$, and
the pressure contains discontinuity in its derivatives along the vertical line $x=0.5$.

\noindent\textbf{Accuracy and robustness test.}
We performed a mesh refinement study with $\nu = 10^{-6}$ and also studied the error behaviors on a fixed mesh while $\nu$ varies for both the 
\texttt{ST-EG} and \texttt{PR-EG} methods. The error behaviors are very similar to those of the previous examples reported in, for example, Table~\ref{test5_h} and Figure~\ref{test5_errors}. Therefore, we omit the results here. 
However, we present the streamlines of the numerical velocity solutions in Figure~\ref{test6_numesol}. The side-by-side comparison of the velocity streamlines generated by the \texttt{ST-EG} and \texttt{PR-EG} methods clearly shows that the \texttt{PR-EG} method well captures the characteristic of the rotational vector field while the \texttt{ST-EG} method does not. 
\begin{figure}[!htb]
\centering
\begin{subfigure}{0.45\linewidth}
\centering
    \includegraphics[width=.75\textwidth]{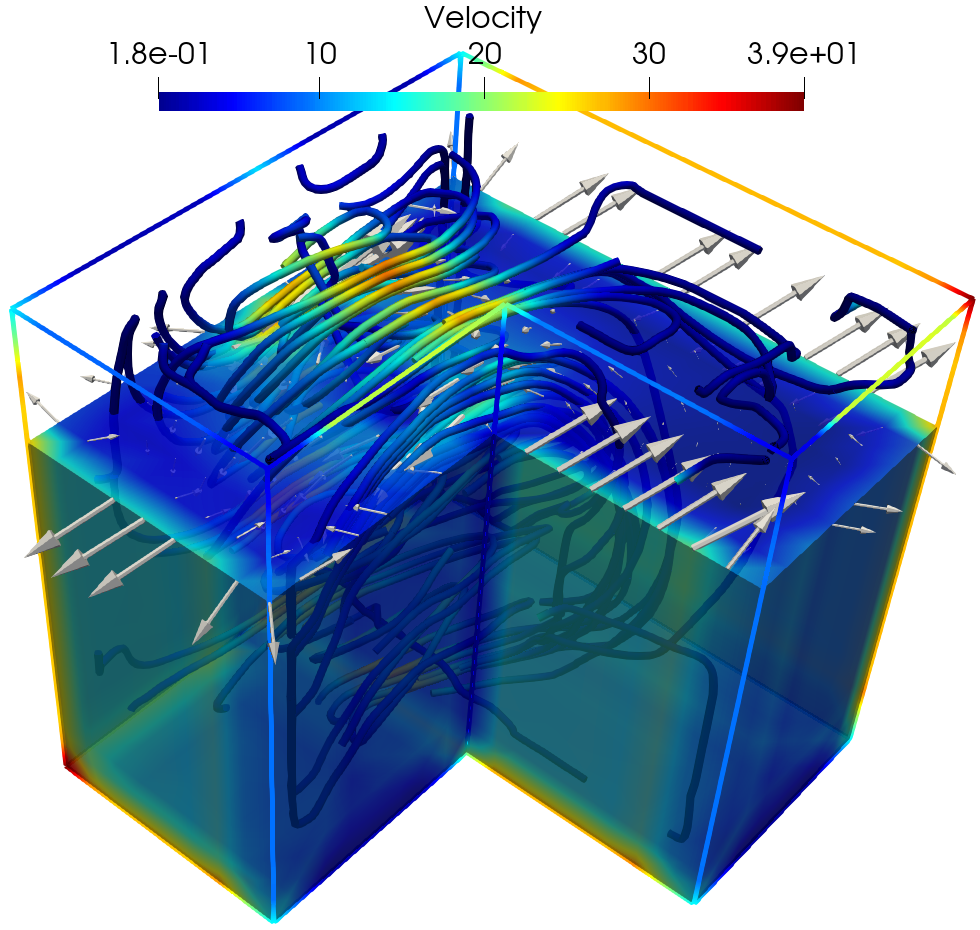} 
    \caption{\texttt{ST-EG}}
    \end{subfigure}
    \hskip 10pt
    \begin{subfigure}{0.45\linewidth}
    \centering
    \includegraphics[width=.75\textwidth]{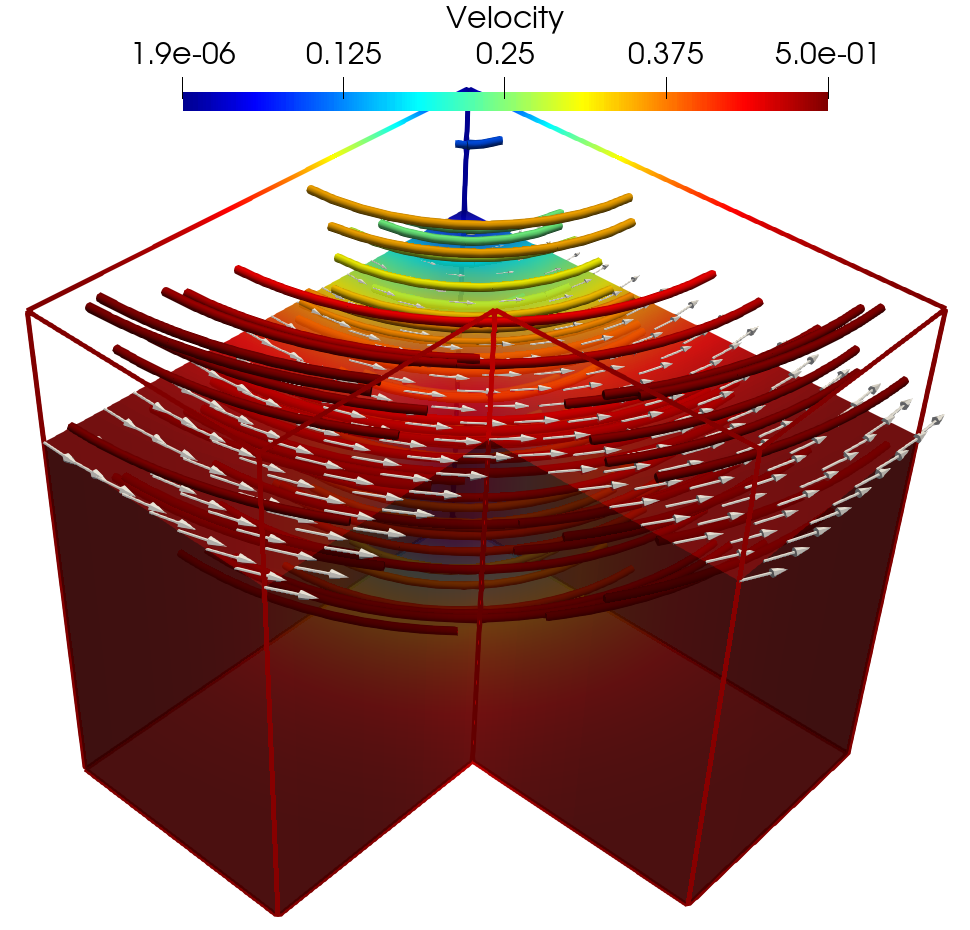}\\ 
    \caption{\texttt{PR-EG}}
    \end{subfigure}
    \caption{Test~\ref{sect:NumTest-6}, 3D Vortex flow in an L-shaped cylinder: Streamlines of the numerical velocity solutions  with $h=1/16$ and $\nu=10^{-6}$.}
    \label{test6_numesol}
\end{figure}

\noindent\textbf{Performance of block preconditioners.}
We again test the performance of the block preconditioners for this 3D L-shaped domain example.
The block preconditioners were implemented in the same way as in Test~\ref{sect:NumTest-5}. The inner GMRES method for solving the diagonal blocks usually took $6$-$9$ iterations in all cases, hence the results are omitted here, but the numbers of iterations of the outer GMRES method are shown in Table~\ref{test4_iter_nu}. When the block preconditioners are used for the \texttt{PR-EG} and \texttt{PPR-EG} methods, the numbers of iterations increase moderately as the viscosity $\nu$ decreases. This is mainly caused by the outer GMRES method since the condition numbers of the preconditioned stiffness matrices remain constant when $\nu$ decreases, as shown in Table~\ref{test4_kappa_nu}. Interestingly, the \texttt{CPR-EG} method performs the best in terms of the number of iterations. However, its condition number is slightly larger than that of the \texttt{PR-EG} method, as shown in Table~\ref{test4_kappa_nu}. Further studies are needed to better understand those observations  to design parameter-robust preconditioners that can be applied in practice. But this is out of the scope of this work and will be part of our future research. 

\begin{table}[!htb]
	\centering
	\begin{tabular}{|c||c|c|c||c|c|c||c|c|c|}
		\hline
		& \multicolumn{9}{c|}{Exact Solver} \\ \hline
		& \multicolumn{3}{c||}{
			\texttt{PR-EG}
		}&
		\multicolumn{3}{c||}{
			\texttt{PPR-EG}
		}
		& \multicolumn{3}{c|}{\texttt{CPR-EG}}\\
		\cline{2-10}
		$\nu$ & $\mathcal{B}_D$ & $\mathcal{B}_L$ & $\mathcal{B}_U$ 
		& $\mathcal{B}^D_D$ & $\mathcal{B}^D_L$ & $\mathcal{B}^D_U$ 
		& $\mathcal{B}^E_D$ & $\mathcal{B}^E_L$ & $\mathcal{B}^E_U$ \\ \hline
		1  &98 &50 &49    &116&62 &59    &64 &33 &31\\ \hline
		$10^{-2}$ &161&85 &85    &207&113&113   &102&56 &56\\ \hline
		$10^{-4}$ &189&101&101   &252&136&136   &105&57 &57\\ \hline
		$10^{-6}$ &217&120&120   &-- &161&161   &105&57 &57\\ \hline
		& \multicolumn{9}{c|}{Inexact Solver} \\ \hline
		& \multicolumn{3}{c||}{
			\texttt{PR-EG}
		}&
		\multicolumn{3}{c||}{
			\texttt{PPR-EG}}
		& \multicolumn{3}{c|}{\texttt{CPR-EG}}
		\\
		\cline{2-10}
		$\nu$ & $\mathcal{M}_D$ & $\mathcal{M}_L$ & $\mathcal{M}_U$
		& $\mathcal{M}^D_D$ & $\mathcal{M}^D_L$ & $\mathcal{M}^D_U$
		& $\mathcal{M}^E_D$ & $\mathcal{M}^E_L$ & $\mathcal{M}^E_U$ \\ \hline
		1   &98 &54 &55    &116&67 &65   &64 &36 &34 \\ \hline
		$10^{-2}$  &161&92 &92    &207&121&121  &102&61 &61 \\ \hline
		$10^{-4}$  &189&112&112 	 &251&147&147  &105&63 &62 \\ \hline
		$10^{-6}$  &211&127&127 	 &279&170&168  &105&63 &62\\ \hline
	\end{tabular}
	\caption{Test~\ref{sect:NumTest-6}, 3D flow in L-shaped domain:  Iteration counts for the block preconditioners when $\nu$ varies on mesh with $h = 1/4$.}
	\label{test4_iter_nu}
\end{table}

\begin{table}[!htb]
	\centering
	\begin{tabular}{|c||c||c||c|}
		\hline
		& \multicolumn{1}{c||}{
			\texttt{PR-EG}
		}&
		\multicolumn{1}{c||}{
			\texttt{PPR-EG}
		}
		& \multicolumn{1}{c|}{\texttt{CPR-EG}}\\
		\cline{2-4}
		$\nu$ & $\kappa(\mathcal{B}_D\mathcal{A})$ 
		& $ \kappa(\mathcal{B}^D_D \mathcal{A}^D)$ 
		& $ \kappa(\mathcal{B}^E_D \mathcal{A}^E)$ \\ \hline
		1  & 130.450  & 267.947 &  164.076 \\ \hline
		$10^{-2}$ & 130.450  & 267.947 &  164.076\\ \hline
		$10^{-4}$ & 130.450  & 267.947 &  164.076\\ \hline
		$10^{-4}$ & 130.450  & 267.947 &  164.076 \\ \hline
	\end{tabular}
	\caption{Test~\ref{sect:NumTest-6}, 3D flow in L-shaped domain:  Condition number of the preconditioned stiffness matrices with the exact block preconditioners when $\nu$ varies on mesh with $h = 1/4$.}
	\label{test4_kappa_nu}
\end{table}

\section{Conclusions}\label{sec:conclusions}
In this paper, we proposed a pressure-robust EG scheme for solving the Stokes equations, describing the steady-state, incompressible viscous fluid flow. The new EG method is based on the recent work \cite{YiEtAl22-Stokes} on a stable EG scheme for the Stokes problem, where the velocity error depends on the pressure error and is inversely proportional to viscosity. In order to make the EG scheme in \cite{YiEtAl22-Stokes} a pressure-robust scheme, we employed a velocity reconstruction operator on the load vector on the right-hand side of the discrete system. Despite this simple modification, our error analysis shows that the velocity error of the new EG scheme is independent of viscosity and the pressure error, and the method maintains the optimal convergence rates for both the velocity and pressure. We also considered a perturbed version of our pressure-robust EG method. This perturbed method allows for the elimination of the DoFs corresponding to the DG component of the velocity vector via static condensation. The resulting condensed linear system can be viewed as an $H^1$-conforming $\mathbb{P}_1$-$\mathbb{P}_0$ scheme with stabilization terms. This stabilized $\mathbb{P}_1$-$\mathbb{P}_0$ scheme is inf-sup stable and pressure-robust as well. Furthermore, we proposed an efficient preconditioning technique whose performance is robust with respect to viscosity. Our two- and three-dimensional numerical experiments verified the theoretical results. In the future, this work will be extended to more complicated incompressible flow models, such as the Oseen and Navier-Stokes equations, where the pressure-robustness is important for simulations in various flow regimes. 

\section*{Acknowlegements}
The work of Son-Young Yi was supported by the U.S. National Science Foundation grant DMS-2208426.

	\bibliography{Stokes}
	
\end{document}